\documentclass[12pt]{amsart}
\usepackage{amsfonts}
\usepackage{amsfonts,latexsym,rawfonts,amsmath,amssymb,amsthm}
\usepackage[plainpages=false]{hyperref}

\usepackage{graphicx}

\RequirePackage{color}

\numberwithin{equation}{section}

\newtheorem{Proposition}{Proposition}[section]
\newtheorem{Theorem}[Proposition]{Theorem}
\newtheorem{Lemma}[Proposition]{Lemma}

\newtheorem{Corollary}[Proposition]{Corollary}

\title  {Global analyticity of affine hyperbolic spheres in the even dimensional space}

\begin{document}
\address{Huaiyu Jian: Department of Mathematical Sciences, Tsinghua University, Beijing 100084, China.}

\address{Xianduo Wang:  Department of Mathematical Sciences, Tsinghua University, Beijing 100084, China.}

\email{hjian@tsinghua.edu.cn;  xd-wang18@mails.tsinghua.edu.cn }

\thanks{This work was supported by NSFC 12141103 }

%\date{}

\bibliographystyle{plain}

%\tableofcontents
\maketitle

\baselineskip=15.8pt
\parskip=3pt

\centerline {\bf Huaiyu Jian \ \ \ \  Xianduo Wang  }

\vskip20pt

\noindent {\bf Abstract}:  A number geometric problems, including affine hyperbolic spheres,
Hilbert metrics and Minkowski type problems, are reduced to a  singular  {M}onge-{A}mp\`ere equation
 which can be written locally  as a class of  {M}onge-{A}mp\`ere equations with singularity
at boundary. We  estimate the boundary derivatives of all orders for the solutions to
the class of singular equations  as possible as optimal, and prove that the solution to the equation is globally analytic
if the dimension of the space is even.
As a corollary, we obtain the  global analyticity of affine hyperbolic spheres in the even dimensional space.
%\begin{abstract}

%\end{abstract}

 \vskip20pt


\vskip20pt

\noindent {\bf  Running head}:     Analyticity of affine hyperbolic spheres
\vskip20pt

\baselineskip=15.8pt
\parskip=3pt

\newpage

\centerline {\bf   Global analyticity of affine hyperbolic spheres in the even dimensional space}

 \vskip10pt

\centerline {Huaiyu Jian \ \ \ \  Xianduo Wang }

%\tableofcontents
\maketitle

\baselineskip=15.8pt
\parskip=3.0pt

\section{Introduciton} %1

In this paper, we study the boundary analyticity of solutions to a class of singular    {M}onge-{A}mp\`ere equations and obtain the 
global analyticity of affine hyperbolic spheres. This is sufficient to    study the global analyticity for the graph of convex solutions to the following Dirichlet problem of a singular Monge-Amp\`ere equation:
\begin{equation}\label{v}
	\left\{ \begin{aligned}
		&\det D^2 v = \frac{1}{|v|^{n+2}} &~\mbox{in}~&\varOmega, \\
		&v = 0 &~\mbox{on}~&\partial\varOmega,
	\end{aligned} \right.
\end{equation}
where $\varOmega$ is a bounded convex domain in $\mathbb{R}^n$ for $n \ge 2$. In fact, 
the Legendre transform of the solution $v$ to \eqref{v} determines a complete affine hyperbolic sphere \cite{MR437805}. Moreover, $(-v)^{-1}\sum v_{x_i x_j} dx_i dx_j$ is a Hilbert metric in the convex domain $\varOmega$ \cite{MR0358078}. Besides, the Minkowski problem in centre-affine geometry and the  $L_p$-Minkowski problem with $p=n+2$   are reduced to the same type equation as \eqref{v}\cite{[CW], [JLW], [Lut]}. See page 441 in \cite{MR3024302}
for the details. 

Affine hyperbolic spheres are a well-known important model in affine geometry \cite{MR0365607, MR859275} as well as a fundamental model in affine sphere relativity \cite{MR3607461}. The existence and uniqueness of solution to \eqref{v} in the space $C^{\infty}(\varOmega)\cap C(\overline{\varOmega})$ were obtained in \cite{MR437805}. However, since the equation \eqref{v} is singular near the boundary, the solution $v$ is only H\"older continuous near $\partial \varOmega$ \cite{MR3771826}. We should mention that the problem about global smoothness and analyticity of affine hyperbolic spheres was once raised by S.T. Yau in a lecture at the Chinese Academy of Sciences early in 2006.

In order to obtain the regularity of affine hyperbolic spheres, it is sufficient to study the regularity of the graph of the solution $v$ to \eqref{v}. Jian and Wang proved in \cite{MR3024302} that the graph of $v$ is $C^{n+2,\alpha}$ up to its boundary for any $\alpha \in (0,1)$. To obtain the boundary regularity, 
for any boundary point $x_0\in \partial \Omega$, without loss of generality assume that $x_0=0 $ and $e_n=(0,\cdots,0,1)$ is the inner normal of $\partial \varOmega$ at $0$. By the rotation of coordinates
\begin{equation*}\begin{aligned}
	&y_n = -x_{n+1},\\
	&y_k = x_k,~1\le k \le n-1,\\
	&y_{n+1} = x_n,
\end{aligned}\end{equation*}
in the new coordinates, the graph of $v$ near the origin can be represented as $y_{n+1}=u(y)$. Then
the problem \eqref{v} can be rewritten as
\begin{equation}\label{u}
	\left\{ \begin{aligned}
		&\det D^2 u = \left( \frac{D_n u}{y_n} \right)^{n+2} &~\mbox{in}~&B_R^+=B_R \cap\{y_n>0\}, \\
		&u = \varphi &~\mbox{on}~&\partial B_R^+ \cap \{y_n=0\},
	\end{aligned} \right.
\end{equation}
where $y = (y', y_n)\in \mathbb{R}^n_+$ and $\{x_n = \varphi(x')\}$ represents a part of $\partial \varOmega$ near $0$ in the original coordinates. 
It was shown  in \cite{MR3024302}   that the solution to \eqref{u} is $C^{n+2, \alpha}$ near $0$ for any $\alpha\in (0,1)$, which yields that affine hyperbolic spheres are globally $C^{n+2, \alpha}$ smooth.

Furthermore, if the dimension $n$ is odd, the solution $u$ to \eqref{u} is at most $C^{n+2, \alpha}$, which was proved in  \cite{MR4361965}  by   expanding   $u$ near $0$. On the other hand, if $n$ is even,  the compatibility condition, given by \cite{MR3024302}, was verified in \cite{MR3705680}, which shows  that the solution to \eqref{u} is globally $C^{\infty}$ smooth. 

In this paper, by the result of smoothness in \cite{MR3705680},  we will prove the analyticity of the solution to \eqref{u} and then obtain global analyticity of affine hyperbolic spheres for the case that $n$ is even. our  main result is 
\begin{Theorem}\label{thm1.1}
	Assume that $n\ge 2$ is even and $\varphi$ is an analytic, uniformly convex function in $\mathbb{R}^{n-1}$. If $u \in C^{\infty}(\overline{\mathbb{R}^n_+})$ is a convex solution to
	\begin{equation}\label{Eq}
		\left\{ \begin{aligned}
			&\det D^2 u = \left( \frac{D_n u}{x_n} \right)^{n+2} &~\mbox{in}~&\mathbb{R}^n_+=\mathbb{R}^n \cap\{x_n>0\}, \\
			&u = \varphi &~\mbox{on}~&\partial \mathbb{R}^n_+,
		\end{aligned} \right.
	\end{equation}
	then $u$ is analytic in $\overline{\mathbb{R}^n_+}$.
\end{Theorem}

Mote that the first equation in \eqref{Eq} is not singular away from the boundary $\{x=(x',x_n):x_n=0\}$. By the standard theory \cite{MR0118970,MR2284971, MR1431263, M1, M2, M3}, we see that the smooth solution $u$ to \eqref{Eq} is analytic away from the boundary. Thus, it is enough to prove the analyticity of $u$ near the boundary.
For this purpose, we define for any $x'_0 \in \mathbb{R}^{n-1}$ and $R>0$ that
\begin{equation*}
	G_R(x'_0) \triangleq \{x=(x',x_n)\in \mathbb{R}^n_+: |x'-x'_0|<R, 0<x_n<R\}
\end{equation*}
and $G_R\triangleq G_R(0')$. Then we only need to find a small $r>0$ such that $u$ is analytic in $\overline{G_r}$ since \eqref{Eq} is invariant under the translation along $x'$-directions.

The key to prove analyticity is the estimating the high-order derivatives about the solution $u$ to \eqref{Eq}. Considering  the singularity of \eqref{Eq} along the $x_n$-direction,  
 Motivated by the ideas from \cite{2018The, 2020The} we classify the derivatives of $u$ into the tangential ones with $D_{x'}$ and the normal ones with $D_n$. Then, in order to obtain high-order derivative estimates, we will make inductive arguments on the order of tangential and normal derivatives respectively. Besides, we will need basic elliptic estimates, related iterations and technical combination equalities or inequalities. 

The paper is organized as follows. In Section \ref{sec2}, we derive some basic and important lemmas as preliminaries. In Section \ref{sec3}, we introduce the linearized problem of \eqref{Eq} and claim the tangential analyticity of solutions to \eqref{Eq}. In Section \ref{tan}, we   finish the proof of the claim in Section \ref{sec3}. In Section \ref{sec5}, we introduce a series of ordinary differential equations 
originated from \cite{MR3024302, MR3705680} and derive some lemmas about derivative estimates. In Section \ref{sec6}, we prove the high-order derivative estimates of solution to \eqref{Eq} and then finish the proof of Theorem \ref{thm1.1}.

%section 2
\section{Preliminaries}\label{sec2}
In this section, we will introduce a nice property of our problem \eqref{Eq} and some basic facts in computations.

First, we refer to \cite{MR3024302} and have the following lemma.

\begin{Lemma}\label{lemma2.1}
	Let $u \in C^2(\overline{\mathbb{R}^n_+})$ be a smooth convex solution to \eqref{Eq}. Then for any $x'\in \mathbb{R}^{n-1}$, we have $D_n u(x',0) =0$. Moreover, for any nonnegative integer $l$, we have $D_{x'}^l D_n u(x',0) =0$.
\end{Lemma}
\begin{proof}
	See Lemma 3.2 in \cite{MR3024302}.
\end{proof}

As a corollary of Lemma \ref{lemma2.1}, it follows that
\begin{Lemma}\label{lemma2.2}
	Let $u \in C^2(\overline{\mathbb{R}^n_+})$ be a smooth convex solution to \eqref{Eq}. Then for any $x'\in \mathbb{R}^{n-1}$,
	\begin{equation}\label{2.1}
		u_{nn}(x',0)=[\det D^2_{x'}\varphi(x')]^{\frac{1}{n+1}}.
	\end{equation}
\end{Lemma}
\begin{proof}
	The proof is almost the same as Lemma 3.5 in \cite{MR3024302}. Fix $x' \in \mathbb{R}^{n-1}$. It follows from Lemma \ref{lemma2.1} that for any integer $1\le i\le n-1$, $u_{in}(x',0)=0$ and
	\[
		\lim_{x_n\to 0} \frac{u_n(x',x_n)}{x_n} =\lim_{x_n\to 0} \frac{u_n(x',x_n)-u_n(x',0)}{x_n}=u_{nn}(x',0).
	\]
	Note that $D_{x'}^2 u(x',0)= D_{x'}^2 \varphi(x',0)$. By \eqref{Eq}, we obtain
	\[
	[u_{nn}(x',0)]^{n+2}=\lim_{x_n\to 0} \det D^2 u(x',x_n)=\det D^2 u(x',0)=
	u_{nn}(x',0)D^2_{x'}\varphi(x',0).
	\]
	Then \eqref{2.1} is proved.
\end{proof}

From the invariance of \eqref{Eq} and uniform convexity of $\varphi$, we assume
\begin{equation}\label{2.2}
	\varphi(0)=0 , D_{x'}\varphi(0)=0, D_{x'}^2 \varphi(0)=D,
\end{equation}
where $D = diag\{\lambda_1,\cdots,\lambda_{n-1}\}$ and $\lambda_i >0 (1\le i \le n-1)$. Denote 
$$\lambda=(\prod_{i=1}^{n-1}\lambda_i)^{-\frac{1}{2(n+1)}}\; \text {and}\; 
  \tilde u(x)=u(D^{-\frac12} x', \lambda x_n).$$  Then by \eqref{Eq}, we have
\begin{equation}\label{2.3}\begin{aligned}
		\det D^2 \tilde{u}(x) &= (\det D)^{-1} \lambda^2 \det D^2 u(D^{-\frac12} x', \lambda x_n)\\
		&= \lambda^{2(n+1)+2} \Big(\frac{D_n u(D^{-\frac12} x', \lambda x_n)}{\lambda x_n}\Big)^{n+2}\\
		&=\lambda^{2(n+2)}\Big(\frac{D_n \tilde{u}(x)}{\lambda^2 x_n}\Big)^{n+2}\\
        & =\Big(\frac{D_n \tilde{u}(x)}{x_n}\Big)^{n+2}
\end{aligned}\end{equation}
and $\tilde{u}(x',0)=\varphi(D^{-\frac12} x')$. By \eqref{2.2}, \eqref{2.3} and Lemma \ref{lemma2.2}, we obtain
\begin{equation}\label{2.4}
	\tilde{u}(0)=0, D\tilde{u}(0)=0, D^2 \tilde{u}(0)=I_n.
\end{equation}
Here and below, $I_k$ denotes the unit matrices of $k$ order. Therefore, with no loss of generality, we can replace $u$ with $\tilde{u}$ and always assume
\begin{equation}\label{2.5}
	u(0)=0, Du(0)=0, D^2 u(0)=I_n
\end{equation}
in the rest of this paper.

Using the notion $l^+=\max\{l, 0\}$ for any integer $l$, we have the following prorperty.
\begin{Lemma}\label{lemma2.3}
	For any integer $k\ge 2$ and nonnegative integers $\lambda, \alpha_1,\cdots,\alpha_k$,
	\begin{align}
		\label{2.6}&(\alpha_1-\lambda)^+ + (\alpha_2-\lambda)^+ \le (\alpha_1+\alpha_2-\lambda)^+;\\
		\label{2.7}&\sum_{i=1}^k (\alpha_i-\lambda)^+ \le \Big(\sum_{i=1}^k \alpha_i-\lambda\Big)^+.
	\end{align}
	Moreover, if $\lambda, \alpha_1, \cdots, \alpha_k$ are positive, we have
	\begin{equation}\label{2.8}
		\sum_{i=1}^k (\alpha_i-\lambda)^+ \le \Big(\sum_{i=1}^k \alpha_i-\lambda-k+1\Big)^+.
	\end{equation}
\end{Lemma}
\begin{proof}
	First, we prove \eqref{2.6} case by case. If $\alpha_1 > \lambda$, we have
	\begin{equation*}
		(\alpha_1-\lambda)^+ + (\alpha_2-\lambda)^+ - (\alpha_1+\alpha_2-\lambda)^+ = \alpha_1-\lambda + (\alpha_2-\lambda)^+ - (\alpha_1+\alpha_2-\lambda)=(\alpha_2-\lambda)^+ -\alpha_2\le 0.
	\end{equation*}
	If $\alpha_2 > \lambda$, it is the same as the first case. If $\alpha_1\le \lambda$ and $\alpha_2 \le \lambda$, it is obvious that
	\[
	(\alpha_1-\lambda)^+ + (\alpha_2-\lambda)^+=0\le (\alpha_1+\alpha_2-\lambda)^+.
	\]
	So \eqref{2.6} is proved, while  \eqref{2.7} is a simple corollary of \eqref{2.6} by induction .
	
	Finally, we prove \eqref{2.8} by \eqref{2.7}. Since $\lambda,\alpha_1,\cdots,\alpha_k$ are positive integers, then $\lambda-1, \alpha_i-1(1\le i\le k)$ are nonnegative integers. Thus by \eqref{2.7}, we obtain
		\begin{equation*}\begin{aligned}
			&\sum_{i=1}^k (\alpha_i-\lambda)^+ =\sum_{i=1}^k [(\alpha_i-1)-(\lambda-1)]^+\\
			\le & \Big[\sum_{i=1}^k(\alpha_i-1)-(\lambda-1)\Big]^+
			=\Big(\sum_{i=1}^k \alpha_i-\lambda-k+1\Big)^+.
	\end{aligned}\end{equation*}
\end{proof}

To be concise, we will directly use Lemma \ref{lemma2.3} without  any citation  in our remaining computation since it  will be   used frequently.

The last lemma in this section is a summary from the proof of Lemma 2.4 in \cite{2020The}.
\begin{Lemma}\label{lemma2.4}
	Denote integers $l\ge 0$, $k\ge 1$ and $\lambda \in \{2,3\}$. Then there exist a constant $S_\lambda>1$ only depending on $\lambda$ such that
	\begin{equation}\label{Sum}
		S[k,l;\lambda]\triangleq\sum_{\substack{\alpha_1+\cdots+\alpha_k=l,\\ \alpha_1,\cdots,\alpha_k\in\mathbb{Z}_0^+}} \frac{(\alpha_1-\lambda)^+!\cdots(\alpha_k-\lambda)^+!}{\alpha_1!\cdots \alpha_k!}
		\le (k S_\lambda)^k \frac{(l-\lambda)^+!}{l!},
	\end{equation}
	where $\mathbb{Z}_0^+$ denotes the set of all nonnegative integers.
\end{Lemma}
\begin{proof}
	At first, Let us consider the case $\lambda=2$. We prove \eqref{Sum} by induction on $l$. Obviously,  it holds for $l=0, 1$ since
	$$S[k,0;2]=1\le k^k,~S[k,1;2]=\sum_{\alpha_1+\cdots+\alpha_k=1} 1=k\le k^k.$$
 Then, we choose an arbitrary integer $p\ge 1$ and assume \eqref{Sum} holds for all $1\le l\le p$ and $k\ge 1$. Now we prove the case for $l = p + 1$. Noting that $S[1,l;2]=\frac{(l-2)^+!}{l!}$, we only need to consider $k\ge 2$. By induction hypotheses, we have
	\begin{equation}\begin{aligned}
		S[k,p+1;2] &\le k \sum_{\substack{\alpha_1+\cdots+\alpha_k=p+1,\\ \alpha_1\ge 1}} \frac{(\alpha_1-2)^+!\cdots(\alpha_k-2)^+!}{(\alpha_1!)\cdots (\alpha_k!)}\\
		&\le k \sum_{\alpha=1}^{p+1} \frac{(\alpha-2)^+!}{\alpha!}S[k-1,p+1-\alpha;2]\\
		&= k \sum_{\alpha=1}^{p} \frac{(\alpha-2)^+!}{\alpha!}S[k-1,p+1-\alpha;2]+k
		\frac{(p-1)!}{(p+1)!}S[k-1,0;2]\\
		&\le 2k \sum_{\alpha=1}^{p} \frac{[(k-1)S_2]^{k-1}}{\alpha^2(p+1-\alpha)^2}
		+\frac{k}{p(p+1)}\\
		&= 2k [(k-1)S_2]^{k-1} \Big(\sum_{\alpha=1}^{[\frac{p+1}{2}]} \frac{1} 
		+\sum_{\alpha=[\frac{p+1}{2}]+1}^{p}\Big) \frac{1}{\alpha^2(p+1-\alpha)^2} +\frac{k}{p(p+1)}\\
		&\le 2k [(k-1)S_2]^{k-1} \frac{4}{(p+1)^2} \Big(\sum_{\alpha=1}^{[\frac{p+1}{2}]} \frac{1}{\alpha^2}
		+\sum_{\alpha=[\frac{p+1}{2}]+1}^{p} \frac{1}{(p+1-\alpha)^2}\Big)+\frac{k}{p(p+1)}\\
		&\le 2k [(k-1)S_2]^{k-1} \frac{8}{(p+1)^2} \sum_{\alpha=1}^{\infty} \frac{1}{\alpha^2}+\frac{k}{p(p+1)}\\
		&=\frac{8\pi^2}{3}\frac{k^k S_2^{k-1}}{(p+1)^2} + \frac{k}{p(p+1)}.
	\end{aligned}\end{equation}
	Hence, we take $S_2 = \frac{8\pi^2}{3} + 2$ and then obtain $S[k, p+1;2]\le\frac{(k S_2)^k}{p(p+1)}$. By induction, \eqref{Sum} is proved for the case $\lambda=2$.
	
	The proof for the case $\lambda=3$ is almost the same. We only need to notice that $\sum\limits_{\alpha=1}^{\infty} \frac{1}{\alpha^3}< +\infty$.
\end{proof}

%section 3
	
\section {  Tangential analyticity and linearized elliptic equations}\label{sec3}

In this section, we will state the tangential analyticity of smooth convex solutions to \eqref{Eq} and then introduce the linearized problem of \eqref{Eq}. First, we claim that if $0<r<1$ is small enough, $u$ is tangentially analytic in $\overline{G_r}$:
\begin{Theorem}\label{TAnalytic}
	Under the assumptions in Theorem \ref{thm1.1}, there exist positive constants $C_0, C_1$ depending on $u, \varphi, r, n$ such that for any nonnegative integer $l$ and $x=(x',x_n)\in \overline{G_r}$,
	\begin{align}
		\label{S1}&|D_{x'}^l (u-\varphi)| \le C_0 C_1^{(l-2)^+} (l-2)^+!(r-|x'|)^{-(l-2)^+} x_n^2,\\
		\label{S2}&|D D_{x'}^l (u-\varphi)| \le C_0 C_1^{(l-2)^+} (l-2)^+!(r-|x'|)^{-(l-2)^+} x_n,\\
		\label{S3}&|D^2 D_{x'}^l (u-\varphi)| \le C_0 C_1^{(l-2)^+} (l-2)^+!(r-|x'|)^{-(l-2)^+}.
	\end{align}
\end{Theorem}

Once Theorem \ref{TAnalytic} is proved, we can obtain derivative estimates of $u$ in $\overline{G_{r/2}}$ from \eqref{S1}-\eqref{S3}. By replacing the value of $r$ by $\frac{r}{2}$, we prove that $u$ and its derivatives up to order $2$ are tangentially analytic in $\overline{G_r}$.

We will prove Theorem \ref{TAnalytic} by induction on the nonnegative integer $l$. First, we prove the case for $0\le l\le 2$. By Lemma \ref{lemma2.1} and Taylor's formula, we have
\begin{align*}
	&D_{x'}^l (u-\varphi)(x',x_n)=\frac{D_n^2 D_{x'}^l (u-\varphi)(x',\theta_1 x_n)}{2} \cdot x_n^2,~~0<\theta_1<1;\\
	&D_n D_{x'}^l (u-\varphi)(x',x_n)=D_n^2 D_{x'}^l (u-\varphi)(x',\theta_2 x_n) \cdot x_n,~~0<\theta_2<1.
\end{align*}
Therefore, there exist positive constants $C_0, C_1$ depending on $\|u-\varphi\|_{C^{4,1}(\overline{G_r})}$ and $r$ such that \eqref{S1}-\eqref{S3} hold for all $0\le l \le 2$ and $x\in G_r$. 

Then, we make {\sl the induction hypotheses for (3.1)-(3.3)}. That is, taking  an arbitrary integer $p\ge 2$ and assume \eqref{S1}-\eqref{S3} hold for all $0\le l \le p$. 
Note that these induction hypotheses will be used again and again in this and next sections. 
By induction, we only need to prove the case for $l=p+1$. Fix $k\in\{1,\cdots,n-1\}$ and denote $w := D_{x_k}(u-\varphi)$. Thus, it is equivalent  to prove that for any $x \in G_r$,
\begin{align}
	\label{SS1}&|D_{x'}^p w| \le C_0 C_1^{(p-1)^+} (p-1)^+!(r-|x'|)^{-(p-1)^+} x_n^2,\\
	\label{SS2}&|D D_{x'}^p w| \le C_0 C_1^{(p-1)^+} (p-1)^+!(r-|x'|)^{-(p-1)^+} x_n,\\
	\label{SS3}&|D^2 D_{x'}^p w| \le C_0 C_1^{(p-1)^+} (p-1)^+!(r-|x'|)^{-(p-1)^+}.
\end{align}

To prove \eqref{SS1}-\eqref{SS3}, we introduce the linearized problem of \eqref{Eq}.
By differentiating \eqref{Eq} in $x_k$, we obtain
\begin{equation}\label{linearized}
	\left\{ \begin{aligned}
		&U^{ij} w_{ij}+P \frac{w_n}{x_n}+N=0~~\mbox{in}~G_r, \\
		&w=0 ~~\mbox{on}~\partial G_r \cap \{x_n=0\},
	\end{aligned} \right.
\end{equation}
where $U^{ij}$ are elements of the cofactor matrix of $D^2 u$, $P=-(n+2)\big(\frac{D_n u}{x_n}\big)^{n+1}$ and $N=\sum_{i,j} U^{ij}\varphi_{kij}$. By the assumption \eqref{2.5} and continuity, we can choose small $0<r<1$ such that
\begin{equation}\label{rassume}
	\frac12\le\big(\frac{D_n u}{x_n}\big)^{n+1}\le\frac32,~~\frac12 I_n\le [U^{ij}]_{1\le i,j\le n}\le\frac32 I_n~~in~G_r.
\end{equation}
It follows that \eqref{linearized} is the Dirichlet problem of an   uniformly elliptic equation about $w$ in $G_r$, but with singular terms $P \frac{w_n}{x_n}$ and $N$. 

Furthermore, we differentiate \eqref{linearized} in $x'$ by $p$ times and obtain
\begin{equation}\label{ptimes}
	\left\{ \begin{aligned}
		&U^{ij} (D_{x'}^p w)_{ij}+P \frac{(D_{x'}^p w)_n}{x_n}+N_p=0~\mbox{in}~G_r, \\
		&D_{x'}^p w=0 ~\mbox{on}~\partial G_r \cap \{x_n=0\},
	\end{aligned} \right.
\end{equation}
where $N_p$ is given by
\begin{equation}\label{Np}
	N_p=D_{x'}^p N + \sum_{m=0}^{p-1} \binom{p}{m}\Big[\sum_{1\le i,j\le n}(D_{x'}^{p-m} U^{ij}) (D_{x'}^m w)_{ij} + D_{x'}^{p-m} P \frac{(D_{x'}^m w)_n}{x_n}\Big].
\end{equation}
Similarly, \eqref{ptimes} is also the Dirichlet problem of an   uniformly elliptic equation about $D_{x'}^p w$ in $G_r$ with singular terms $P \frac{(D_{x'}^p w)_n}{x_n}$ and $N_p$.  
 Moreover. it is easy to see that \eqref{SS1}-\eqref{SS3} are actually estimates of $D_{x'}^p w$ and its derivatives up to order 2.

 We prove \eqref{SS1}-\eqref{SS3} in the next section. As preparation, 
 in the remains of this section we will give  the  estimates for tangential derivatives  of the coefficients $U^{ij}$ and $P $  and the 
 estimate  of the coefficient$N_p$ in \eqref{ptimes}.
\begin{Lemma}\label{lemma3.2}
	Assume the  induction hypotheses  for (3.1)-(3.3).  There exist positive constants $C_0, C_1, \tilde{C_0}$ depending on $u,\varphi,r,n$ such that  the following estimates always hold for any $x\in G_r$ and integer $0\le l\le p$:
	\begin{align}
		\label{Uij}&|D_{x'}^l U^{ij}(x)| \le \tilde{C_0} C_1^{(l-2)^+} (l-2)^+!(r-|x'|)^{-(l-2)^+};\\
		\label{P}&|D_{x'}^l P(x)| \le \tilde{C_0} C_1^{(l-2)^+} (l-2)^+!(r-|x'|)^{-(l-2)^+}.
	\end{align}
\end{Lemma}
\begin{proof}
	(I) At first, by the analyticity of $\varphi$, we can choose proper $C_0, C_1$ such that for any integer $m \ge 2$ and $x\in\overline{G_r}$,
	\begin{equation}\label{Dvarphi}
		|D_{x'}^m \varphi| \le C_0 (C_1/r)^{(m-5)^+} (m-5)^+! \le C_0 (C_1/r)^{(m-4)^+} (m-4)^+!.
	\end{equation}
	For any $0\le l\le p$, by \eqref{Dvarphi} and the induction hypothesis for \eqref{S3}, we have
	\begin{equation}\label{D2u}
		|D^2 D_{x'}^l u| \le 2C_0 C_1^{(l-2)^+} (l-2)^+!(r-|x'|)^{-(l-2)^+}.
	\end{equation}
	By the definition of $U^{ij}$, we obtain
	\begin{equation}
		|D_{x'}^l U^{ij}| \le \sum |D_{x'}^l (u_{i_1 j_1}u_{i_2 j_2}\cdots u_{i_{n-1} j_{n-1}})|,
	\end{equation}
	where $(i_1,\cdots,i_{n-1})=(1,\cdots,i-1,\widehat{i},i+1,\cdots,n)$ and $(j_1,\cdots,j_{n-1})$  runs every permutation of the set $\{1,\cdots,j-1,\widehat{j},j+1,\cdots,n\}$.
	Now we fix $i,j,i_q, j_q(1\le q\le n-1)$ and only consider $D_{x'}^l (u_{i_1 j_1}\cdots u_{i_{n-1} j_{n-1}})$. Then by \eqref{D2u} and Lemma \ref{lemma2.4}, we obtain that for any $0\le l \le p$ and $x\in G_r$,
	\begin{equation}\begin{aligned}
		&|D_{x'}^l (u_{i_1 j_1}u_{i_2 j_2}\cdots u_{i_{n-1} j_{n-1}})| = \sum_{\alpha_1+\cdots+\alpha_{n-1}=l} \frac{l!}{\alpha_1!\cdots \alpha_{n-1}!} |D_{x'}^{\alpha_1}u_{i_1 j_1}|\cdots |D_{x'}^{\alpha_{n-1}}u_{i_{n-1} j_{n-1}}| \\
		&\le \sum_{\alpha_1+\cdots+\alpha_{n-1}=l} \frac{l!(\alpha_1-2)^+!\cdots(\alpha_{n-1}-2)^+!}{\alpha_1!\cdots \alpha_{n-1}!}
		(2C_0)^{n-1}[C_1/(r-|x'|)]^{(\alpha_1-2)^+\cdots+(\alpha_{n-1}-2)^+}\\
		&\le (2C_0)^{n-1}[C_1/(r-|x'|)]^{(l-2)^+} l!
		\sum_{\alpha_1+\cdots+\alpha_{n-1}=l} \frac{(\alpha_1-2)^+!\cdots(\alpha_{n-1}-2)^+!}{\alpha_1!\cdots \alpha_{n-1}!}\\
		&= (2C_0)^{n-1}[C_1/(r-|x'|)]^{(l-2)^+} l!
		S[n-1,l;2]\\
		&\le [2(n-1)C_0 S_2]^{n-1}[C_1/(r-|x'|)]^{(l-2)^+} (l-2)^+!.
	\end{aligned}\end{equation}
	Therefore,
	\begin{equation*}
		|D_{x'}^l U^{ij}|\le (n-1)![2(n-1)C_0 S_2]^{n-1}[C_1/(r-|x'|)]^{(l-2)^+} (l-2)^+!.
	\end{equation*}
	
	(II) We note that $u_n:=D_nu=D_n(u-\varphi) $. Similarly as (I), by the induction hypothesis for \eqref{S2} and Lemma \ref{lemma2.4}, we obtain that for any $0\le l \le p$ and $x\in G_r$,
	\begin{equation}\begin{aligned}
		\Big|\frac{1}{n+2} D_{x'}^l P\Big| &\le \sum_{\alpha_1+\cdots+\alpha_{n+1}=l} \frac{l!}{(\alpha_1!)\cdots (\alpha_{n+1}!)} \frac{|D_{x'}^{\alpha_1}u_n|}{x_n}\cdots \frac{|D_{x'}^{\alpha_{n+1}}u_n|}{x_n}\\
		&\le (C_0)^{n+1} [C_1/(r-|x'|)]^{(l-2)^+} l! S[n+1, l; 2]\\
		&\le [(n+1)C_0 S_2]^{n+1} [C_1/(r-|x'|)]^{(l-2)^+} (l-2)^+!.
	\end{aligned}\end{equation}
\end{proof}
\begin{Lemma}\label{lemma3.3}
	Under the same assumption  as Lemma \ref{lemma3.2}, there exists a positive constant $C_2$ depending on $u, \varphi, r, n$ such that for any $x\in G_r$,
	\begin{equation}
		|N_p| \le C_2 C_1^{p-2} (p-1)! (r-|x'|)^{-(p-2)}.
	\end{equation}
\end{Lemma}
\begin{proof}
	By \eqref{Np}, we write $$N_p \triangleq I_1 + I_2 + I_3$$ and estimate each term  respectively for any given $x\in G_r$.
		
	(1) It follows from \eqref{Dvarphi}, (3.11) and  Lemma  2.4 that
	
	\begin{equation}\begin{aligned}
		|I_1| &= |D_{x'}^p N| = |\sum_{i,j=1}^n D_{x'}^p (U^{ij}\varphi_{kij})|\\
		&\le \sum_{i,j=1}^n \sum_{m=0}^{p} \binom{p}{m} |D_{x'}^m U^{ij}| |D_{x'}^{p-m+3} \varphi|\\
		&\le n^2 \sum_{m=0}^{p} \frac{p!(m-2)^+!(p-m-2)^+!}{m!(p-m)!} \tilde{C_0}C_0[C_1/(r-|x'|)]^{(m-2)^+ +(p-m-2)^+}\\
		&\le n^2 \tilde{C_0} C_0 [C_1/(r-|x'|)]^{p-2} p! \sum_{m=0}^{p} \frac{(m-2)^+!(p-m-2)^+!}{m!(p-m)!}\\
		&=n^2 \tilde{C_0} C_0 [C_1/(r-|x'|)]^{p-2} p! S[2, p;2]\\
		&\le 4 n^2 S_2^2 \tilde{C_0} C_0 [C_1/(r-|x'|)]^{p-2} (p-2)!\\
		&\le 4 n^2 S_2^2 \tilde{C_0} C_0 [C_1/(r-|x'|)]^{p-2} (p-1)!.
	\end{aligned}\end{equation}
		
	(2) By induction hypotheses for \eqref{S3} and Lemma \ref{lemma3.2}, we have
	\begin{equation}\label{3.20}\begin{aligned}
		|I_2|&\le \sum_{i,j=1}^{n}\sum_{m=1}^{p}\binom{p}{m}|D_{x'}^{m}U^{ij}||D_{x'}^{p-m}(u-\varphi)_{kij}|\\
		&\le n^2 \tilde{C_0} C_0 \sum_{m=1}^{p} [C_1/(r-|x'|)]^{(m-2)^+ +(p-m+1-2)^+} \cdot \frac{p!(m-2)^+!(p-m+1-2)^+!}{m!(p-m)!}\\
		&= n^2 \tilde{C_0} C_0 \sum_{m=1}^{p} [C_1/(r-|x'|)]^{[(m-1)-1]^+ +[(p-m)-1]^+} \cdot \frac{p!(m-2)^+!(p-m-1)^+!}{m!(p-m)!}\\
		&\le n^2 \tilde{C_0} C_0 [C_1/(r-|x'|)]^{p-2} \sum_{m=1}^{p} \frac{p!(m-2)^+!(p-m-1)^+!}{m!(p-m)!}.
	\end{aligned}\end{equation}
	Here, we claim that for any $p \ge 2$,
\begin{equation}\label{3.21}
	 \sum_{m=1}^{p} \frac{p!(m-2)^+!(p-m-1)^+!}{m!(p-m)!}\le 5(p-1)!.
	\end{equation} 

In fact, if $p=2$, it is easy to verify
	\[\sum_{m=1}^{p} \frac{p!(m-2)^+!(p-m-1)^+!}{m!(p-m)!} = 3 <5 = 5(p-1)!.\]
	If $p\ge 3$, we have
	\begin{equation*}\begin{aligned}
		&\sum_{m=1}^{p} \frac{p!(m-2)^+!(p-m-1)^+!}{m!(p-m)!} \\
		=& \frac{p!}{p-1}+(p-2)! +\sum_{m=2}^{p-1} \frac{p!}{m(m-1)(p-m)}\\
		\le& 2(p-1)!+(p-1)!+\sum_{m=2}^{p-1} \frac{(p-1)!}{m-1}\cdot\frac{p}{m(p-m)}\\
		=& (p-1)!\Big[3+\sum_{m=2}^{p-1} \frac{1}{m-1}(\frac{1}{m}+\frac{1}{p-m})\Big]\\
		\le& (p-1)!\Big[3+ \sum_{m=2}^{p-1}(\frac{1}{m-1}-\frac{1}{m}) + \sum_{m=2}^{p-1} \frac{1}{p-2}\Big]\\
		\le& 5(p-1)!.
	\end{aligned}\end{equation*}

	Hence, by  \eqref{3.20} and (3.21), we obtain
	\begin{equation}
		|I_2|\le 5 n^2 \tilde{C_0} C_0 [C_1/(r-|x'|)]^{p-2} (p-1)!.
	\end{equation}
	
	(3) Similarly as (2), we use (3.12), the induction hypothesis for \eqref{S2} and (3.21) to obtain
	
	\begin{equation}\begin{aligned}
		|I_3|&\le \sum_{m=1}^p \binom{p}{m} |D_{x'}^m P|\Big|\frac{D_n D_{x'}^{p-m} (u-\varphi)_k}{x_n}\Big|\\
		&\le \tilde{C_0} C_0 \sum_{m=1}^{p} [C_1/(r-|x'|)]^{(m-2)^+ +(p-m+1-2)^+} \cdot \frac{p!(m-2)^+!(p-m+1-2)^+!}{m!(p-m)!}\\
		&\le 5 \tilde{C_0} C_0 [C_1/(r-|x'|)]^{p-2} (p-1)!.
	\end{aligned}\end{equation}
	
	Since $|N_p|\le |I_1|+|I_2|+|I_3|$, we complete the proof of Lemma 3.3.
\end{proof}

\section{proof of Theorem \ref{TAnalytic}}\label{tan}
	
In this section, we will finish the proof of Theorem \ref{TAnalytic} under the induction hypotheses for (3.1)-(3.3) in the last section. As shown in Section \ref{sec3}, it is sufficient to prove that \eqref{SS1}-\eqref{SS3} hold for all $x\in G_r$.  motivated by  the proof of Theorem 2.1 in \cite{2018The} we will  divide our proof into four parts.  As the notations such as $c_0, c_1$ and $c_2$,  
 we will use the notations $c_i (0\le i\le 7) $ to represent positive constants only depending on $u, \varphi, r, n$.
	
{\bfseries Part 1. }\;  We first prove \eqref{SS1} for any $x\in G_r$. For convenience, we define an operator $L$ by $$Lw\triangleq U^{ij}w_{ij}+P\frac{w_n}{x_n}.$$
For any fixed $x'_0 \in B'_r(0)$, we denote $\rho = \frac{1}{p} (r-|x'_0|)$ and $\bar{w}(x)=A|x'-x'_0|^2+Bx_n^2$, where $A, B$ are positive constants to be determined later.  By computation, we have
\begin{align}
	\label{4.1}&L\bar{w}=2A\sum_{i<n} U^{ii}+2B(U^{nn}+P)~in~G_\rho(x'_0),\\
	&\bar{w}\ge 0~on~\partial G_\rho(x'_0)\cap\{x_n=0\},\\
	&\bar{w}\ge B\rho^2~on~\partial G_\rho(x'_0)\cap\{x_n=\rho\},\\
	&\bar{w}\ge A\rho^2~on~\partial G_\rho(x'_0)\cap\{0<x_n<\rho\}.
\end{align}
It follows from \eqref{4.1} and \eqref{rassume} that
\begin{equation}
	L\bar{w}\le 2A(n-1)\frac{3}{2}+2B[\frac{3}{2}-\frac 12 (n+2)]=-(n-1)(B-3A).
\end{equation}
Note that $|x'|\le |x'-x'_0|+|x'_0|< \rho+|x'_0|$ and $0< x_n < \rho$ for any $x\in G_{\rho}(x'_0)$. By induction hypotheses for \eqref{S2}, we obtain
\begin{equation}\label{4.6}\begin{aligned}
		&|D_{x'}^p w|\leq |D D_{x'}^{p} (u-\varphi)|\le C_0 C_1^{p-2} (p-2)!(r-|x'|)^{-(p-2)} x_n\\
		&\le C_0 C_1^{p-2} (p-2)!(r-|x'_0|-\rho)^{-(p-2)}\cdot\rho\\
		&=C_0 C_1^{p-2} (p-2)! (1-\frac{1}{p})^{-(p-2)} (r-|x'_0|)^{-(p-2)}\cdot\rho^{-1}\cdot\rho^2\\
		&=C_0 C_1^{p-2} (p-2)! (1+\frac{1}{p-1})^{p-2} (r-|x'_0|)^{-(p-2)}\cdot p (r-|x'_0|)^{-1} \cdot\rho^2\\
		&=C_0 C_1^{p-2} (p-1)! (1+\frac{1}{p-1})^{p-1} (r-|x'_0|)^{-(p-1)}\cdot\rho^2\\
		&\le e C_0 C_1^{p-2} (p-1)!(r-|x'_0|)^{-(p-1)}\cdot\rho^2~~in~G_\rho(x'_0).
\end{aligned}\end{equation}

Similarly, by Lemma \ref{lemma3.3}, we have
\begin{equation}\label{4.7}
	|N_p|\le C_2 C_1^{p-2}(p-1)!(r-|x'_0|-\rho)^{-(p-2)}\le e C_2C_1^{p-2}(p-1)!(r-|x'_0|)^{-(p-2)}
\end{equation}
in $G_{\rho}(x'_0)$.
Now we choose $A=e C_0 C_1^{p-2} (p-1)!(r-|x'_0|)^{-(p-1)}$ and  $B\geq A$ such that $$(n-1)(B-3A)=e C_2C_1^{p-2}(p-1)!(r-|x'_0|)^{-(p-2)}.$$
Therefore, combining \eqref{ptimes} with \eqref{4.1}-\eqref{4.7}, we have obtained that $\bar{w}$ (or $-\bar{w}$) is a supersolution (or subsolution) of $D_{x'}^p w$ with respect to the operator $L$ in $G_{\rho}(x'_0)$. So if $t < \rho$, we have $(x'_0,t)\in \overline{G_{\rho}(x'_0)}$ and
\begin{equation}\label{4.8}
	|D_{x'}^p w(x'_0,t)| \le \bar{w}(x'_0,t)=Bt^2\le e(3C_0+\frac{C_2}{n-1})C_1^{p-2}(p-1)!(r-|x'_0|)^{-(p-1)}t^2.
\end{equation}
From the proof of Lemma 3.3 we see that the value of $C_2$ only depends on $C_0$ and $\tilde{c_0}$. Hence, taking $C_1 \ge e(3+\frac{C_2}{(n-1)C_0})$  we obtain from \eqref{4.8} that
\begin{equation}
	|D_{x'}^p w(x'_0,t)| \le C_0 C_1^{p-1}(p-1)!(r-|x'_0|)^{-(p-1)}t^2.
\end{equation}
On the other hand, if $t \ge \rho = \frac{1}{p}(r-|x'_0|)$, by the induction hypothesis for (3.2) we also have
\begin{equation}\begin{aligned}
	&|D_{x'}^p w(x'_0,t)|\leq |D D_{x'}^{p} (u-\varphi)(x'_0,t)|
	\le C_0 C_1^{p-2} (p-2)!(r-|x'_0|)^{-(p-2)} \frac{1}{t}t^2 \\
	&\le C_0 C_1^{p-2}p(p-2)!(r-|x'_0|)^{-(p-1)}t^2
	= C_0 C_1^{p-1}(p-1)!(r-|x'_0|)^{-(p-1)}t^2,
\end{aligned}\end{equation}
where we have used the fact $p\le 2(p-1)$ and  $C_1 \ge 2$.
 In this way, we finish the proof of \eqref{SS1}.

{\bfseries Part 2.} \; In this part, we prove \eqref{SS2} for any $x\in G_r$. For any given $x_0=(x'_0,t)\in G_r$, we assume $\rho = t$. If $\rho \ge \frac{1}{p}(r-|x'_0|)$, we obtain by 
the induction hypothesis for \eqref{S3} that
\begin{equation}\begin{aligned}
	|DD_{x'}^{p}w(x_0)| &\leq |D^2 D_{x'}^p (u-\varphi)(x_0)|\\
	&\le C_0 C_1^{p-2} (p-2)! (r-|x'_0|)^{-(p-2)}\frac{1}{\rho} t\\
	&\le C_0 C_1^{p-2} \frac{p}{p-1} (p-1)! (r-|x'_0|)^{-(p-1)} t\\
	&\le 2C_0 C_1^{p-2} (p-1)! (r-|x'_0|)^{-(p-1)} t.
\end{aligned}\end{equation}
Taking $C_1>2$, we see that \eqref{SS2} holds for all $x_0=(x'_0,t) \in G_r$ such that $t=\rho \ge \frac{1}{p}(r-|x'_0|)$. Now we assume $\rho < \frac{1}{p}(r-|x'_0|)$. Then we have $B_{\rho}(x_0) \subset G_r$. Denote $v^p=\frac{1}{\rho^2}(D_{x'}^p w)(\rho x + x_0)$. It follows from \eqref{ptimes} that $v^p$ satisfies
\begin{equation}\label{scalingEq}
	U^{ij}(\rho x + x_0) v^p_{ij}(x) + \rho\frac{P(\rho x + x_0)}{\rho x_n+t} v^p_n(x) + N_p(\rho x + x_0)=0~in~B_{3/4}(0).
\end{equation}
Note that
\begin{equation}
	\|U^{ij}\|_{L^{\infty}(B_{3\rho/4}(x_0))}+\|\rho \frac{P}{x_n}\|_{L^{\infty}(B_{3\rho/4}(x_0))}\le c_0,
\end{equation}
which, together with (3.8) and the local $C^{1,\alpha}$ estimates for uniformly elliptic equations, implies that for any given $\alpha\in(0,1)$, 
\begin{equation}\label{4.14}     %(4.14)
	\|v^p\|_{C^{1,\alpha}(B_{1/2}(0))}\le c_1 (\|v^p\|_{L^{\infty}(B_{3/4}(0))}+\|N_p\|_{L^{\infty}(B_{3\rho/4}(x_0))}).
\end{equation}
By the definition of $v^p$, scaling back from \eqref{4.14}, we have
\begin{equation}\label{D1}\begin{aligned}
	&\rho^{-1} \|DD_{x'}^p w\|_{L^{\infty}(B_{\rho/2}(x_0))} + \rho^{\alpha-1}
	[DD_{x'}^p w]_{C^{\alpha}(B_{\rho/2}(x_0))}\\
	&\le c_1(\rho^{- 2}\|D_{x'}^p w\|_{L^{\infty}(B_{3\rho/4}(x_0))}+\|N_p\|_{L^{\infty}(B_{3\rho/4}(x_0))}).
\end{aligned}\end{equation}
Combining the results from the   Part 1 (see \eqref{4.7} and \eqref{4.8}), we obtain
\begin{equation}\begin{aligned}
	|DD_{x'}^p w(x_0)|&\le c_1 \Big[\rho^{-1} e(3C_0+\frac{C_2}{n-1}) C_1^{p-2}(p-1)!
	(r-|x'_0|-3\rho/4)^{-(p-1)}(\rho+3\rho/4)^2 \\
	&+\rho C_2 C_1^{p-2}(p-1)!(r-|x'_0|-3\rho/4)^{-(p-2)}\Big].
\end{aligned}\end{equation}
Noting that $t=\rho < \frac{1}{p} (r-|x'_0|)$, we deduce that
\begin{equation}
	|DD_{x'}^p w(x_0)| \le C_0 C_3 C_1^{p-2}(p-1)!(r-|x'_0|)^{-(p-1)}t.
\end{equation}
Hence, taking $C_1\ge C_3$  we have proved that \eqref{SS2} also holds for all $x_0=(x'_0, t)\in G_r$ such that $t = \rho <\frac{1}{p} (r-|x'_0|)$.

{\bfseries Part 3.} In this part, we prove \eqref{SS3} for all $x=(x',x_n) \in G_r$ such that $x_n < \frac{1}{p} (r-|x'|)$. Fix any $\alpha\in(0,1)$ and $x_0=(x'_0,t)\in G_r$. We assume $\rho=t<\frac{1}{p} (r-|x'_0|)$. Note that
\begin{equation}\begin{aligned}
	&[U^{ij}(\rho x+x_0)]_{C^{\alpha}(B_{1/2}(0))}+\Big[\rho\frac{P(\rho x + x_0)}{\rho x_n+t}\Big]_{C^{\alpha}(B_{1/2}(0))}\\
	=&\rho^{\alpha}[U^{ij}]_{C^{\alpha}(B_{\rho/2}(x_0))}+\rho^{\alpha}\Big[\frac{P}{x_n/\rho}\Big]_{C^{\alpha}(B_{\rho/2}(x_0))}\le c_2.
\end{aligned}\end{equation}
Thus, applying the  local $C^{2,\alpha}$ estimates to    the  uniformly elliptic equations  \eqref{scalingEq}, we obtain
\begin{equation}\label{C2alpha}
	\|v^p\|_{C^{2,\alpha}(B_{1/4}(0))}\le c_3 (\|v^p\|_{L^{\infty}(B_{1/2}(0))}
	+\|N_p(\rho x+x_0)\|_{C^{\alpha}(B_{1/2}(0))}).
\end{equation}
By the definition of $v^p$, scaling back from \eqref{C2alpha}, we have
\begin{equation}\label{D2}\begin{aligned}
	&|D^2 D_{x'}^p w(x_0)| = |D^2 v^p(0)| \le \|v^p\|_{C^{2,\alpha}(B_{1/4}(0))} \\
	\le &c_3 (\rho^{-2}\|D_{x'}^p w\|_{L^{\infty}(B_{\rho/2}(x_0))}
	+\|N_p\|_{L^{\infty}(B_{\rho/2}(x_0))}+\rho^\alpha [N_p]_{C^{\alpha}(B_{\rho/2}(x_0))}).
\end{aligned}\end{equation}
Sine we have (4.7) and Part 1,  this leads us  to estimate the term $\rho^\alpha [N_p]_{C^{\alpha}(B_{\rho/2}(x_0))}$. For general $l\le p$, by the definition of \eqref{Np}, we have
\begin{equation}\label{Nl}
	N_l=D_{x'}^l N + \sum_{m=0}^{l-1} \binom{l}{m}\Big[\sum_{1\le i,j\le n}(D_{x'}^{l-m} U^{ij}) (D_{x'}^m w)_{ij} + D_{x'}^{l-m} P \frac{(D_{x'}^m w)_n}{x_n}\Big],
\end{equation}
which shows that $N_l$ contains such terms as $\frac{D D_{x'}^{m} w}{x_n}, D^2 D^{m-1}_{x'} w$ for $m\le l$. Noting that $x_n > \frac{\rho}{2}$ for $x=(x',x_n)\in B_{\rho/2(x_0)}$, we can estimate $\frac{D D_{x'}^{m} w}{x_n}$ by
\begin{equation}
	\rho^{\alpha}[\frac{D D_{x'}^{m}w}{x_n}]_{C^{\alpha}(B_{\rho/2}(x_0))}
	\le 2 (\|D D_{x'}^{m} w\|_{L^{\infty}(B_{\rho/2}(x_0))}
	+ \rho^{\alpha-1}[D D_{x'}^{m} w]_{C^{\alpha}(B_{\rho/2}(x_0))}).
\end{equation}
Moreover, $D^2 D^{m-1}_{x'} w$ can be regarded as $D D^{m}_{x'} w$ or $D_n^2 D^{m-1}_{x'} w$. Repeating the proof in Part 2, we can obtain the estimates $D D^{m}_{x'} w$ for $m\le l$. As for $D_n^2 D^{m-1}_{x'} w$, we claim that for any $m \le l$,
\begin{equation}\label{4.23}
	\rho^\alpha [D_n^2 D_{x'}^{m} u]_{C^{\alpha}(B_{\rho/2}(x_0))} \le c_4 C_0 C_1^{(m-2)^+} (m-2)^+! (r-|x'_0|)^{-(m-2)^+}.
\end{equation}
We will prove the claim by induction. By the smoothness of $u$ and $\varphi$, we can assume \eqref{4.23} holds for $m = 0,1,2$. Then, for any $2\le l \le p-1$, we assume \eqref{4.23} holds for all $0 \le m\le l$. Similarly as Lemma \ref{lemma3.3}, it follows from \eqref{Nl} that
\begin{equation}
	\|N_l(\rho x+x_0)\|_{C^{\alpha}(B_{1/2}(0))} \le c_4 C_0 C_1^{l-2} (l-1)! (r-|x'_0|)^{-(l-1)}.
\end{equation}
Denote $v^l=\frac{1}{\rho^2}(D_{x'}^l w)(\rho x + x_0)$.  It follows from the equation in (3.9) that $v^l$ satisfies
\begin{equation}
	U^{ij}(\rho x + x_0) v^l_{ij}(x) + \rho\frac{P(\rho x + x_0)}{\rho x_n+t} v^l_n(x) + N_l(\rho x + x_0)=0~in~B_{3/4}(0),
\end{equation}
which yields 
\begin{equation}\label{Dnn}
	v^l_{nn}(x)=-\frac{1}{U^{nn}(\rho x+x_0)}\Big(\sum_{i+j<2n} U^{ij}(\rho x+x_0)v^l_{ij}(x)+\rho\frac{P(\rho x + x_0)}{\rho x_n+t} v^l_n(x) + N_l(\rho x + x_0)\Big).
\end{equation}
By \eqref{Dnn} and the induction hypothesis for (4.23), we can prove
\begin{equation}
	\|v^l_{nn}\|_{C^{\alpha}(B_{1/2}(0))} \le c_4 C_0 C_1^{l-2} (l-1)! (r-|x'_0|)^{-(l-1)}.
\end{equation}
By the definition of $v^l$, we obtain \eqref{4.23} for $m=l+1$. By induction, we finish the proof of our claim. As a result, we can use the same method as in
the proof of  Lemma \ref{lemma3.3} to prove
\begin{equation}
	\|N_p(\rho x+x_0)\|_{C^{\alpha}(B_{1/2}(0))} \le \tilde{c_4} C_0 C_1^{p-2} (p-1)! (r-|x'_0|)^{-(p-1)}.
\end{equation}
Therefore, from (4.20) and the  discussion below \eqref{D2}, we obtain
\begin{equation}
	|D^2 D_{x'}^p w(x_0)| \le c_5 C_0 C_1^{p-2} (p-1)! (r-|x'_0|)^{-(p-1)}\le C_0 C_1^{p-1} (p-1)! (r-|x'_0|)^{-(p-1)}
\end{equation}
if we take $C_1 \ge c_5$.

{\bfseries Part 4.} \; In this part, we prove \eqref{SS3} for all $x=(x',x_n) \in G_r$ such that $x_n \geq  \frac{1}{p} (r-|x'|)$.   For this purpose, we fix an arbitrary $\alpha \in (0,1)$ and $x_0=(x'_0, t)\in G_r$ such that $t\ge \frac{1}{p}(r-|x'_0|)$. We may assume $t \le \frac{r}{2}$. Let $\rho = \frac{1}{2p}(r-|x'_0|)$. Then we have $B_\rho(x_0) \subset B_r(0)$. Denote $v^{p+1}=\frac{1}{\rho^2}(D^{p+1}_{x'} w)(\rho x + x_0)$. As the proof of \eqref{4.14}, we obtain
\begin{equation}
	\|v^{p+1}\|_{C^{1,\alpha}(B_{1/2}(0))}\le c_6 (\|v^{p+1}\|_{L^{\infty}(B_1(0))}+\|N_{p+1}\|_{L^{\infty}(B_{\rho}(x_0))}),
\end{equation}
which yields
\begin{equation}\label{4.31}
	|DD_{x'}^{p+1}w(x_0)|\le c_6 (\rho^{-1}\|D_{x'}^{p+1}w\|_{L^{\infty}(B_{\rho}(x_0))}+\rho\|N_{p+1}\|_{L^{\infty}(B_{\rho}(x_0))}).
\end{equation}
By the induction hypothesis  for (3.3),   we have
\begin{equation}\label{4.32}\begin{aligned}
	&\|D_{x'}^{p+1}w\|_{L^{\infty}(B_{\rho}(x_0))}\le\|D^2D_{x'}^{p-1}w\|_{L^{\infty}(B_{\rho}(x_0))}\\
	\le &C_0 C_1^{p-2} (p-2)! (r-|x'_0|-\rho)^{-(p-2)}\le eC_0 C_1^{p-2} (p-2)! (r-|x'_0|)^{-(p-2)}.
\end{aligned}\end{equation}
Now it remains to estimate $N_{p+1}$. It follows from the proof of Lemma \ref{lemma3.2} that in $B_\rho(x_0)$,
\begin{equation}\begin{aligned}
	&|D_{x'}^{p+1}U^{ij}|\le \sum |D_{x'}^{p+1} (u_{i_1 j_1}u_{i_2 j_2}\cdots u_{i_{n-1} j_{n-1}})|
	\le (n-1)! \Bigg(\sum_{\substack{\alpha_1+\cdots+\alpha_{n-1}=p+1,\\ \alpha_1,\cdots, \alpha_{n-1}< p+1}}  \\
	&\frac{(p+1)!}{(\alpha_1!)\cdots (\alpha_{n-1}!)} |D^2 D_{x'}^{\alpha_1}u|\cdots |D^2 D_{x'}^{\alpha_{n-1}}u|
	+(n-1)C_0^{n-2}|D^2D_{x'}^{p+1}u|\Bigg)\\
	&\le \tilde{C_0} \Big(C_1^{p-1}(p-1)!(r-|x'|)^{-(p-1)}+\|D^2D_{x'}^{p}w\|_{L^{\infty}(B_{\rho}(x_0))}\Big)
\end{aligned}\end{equation}
and
\begin{equation}\begin{aligned}
	|(n+2)^{-1} &D_{x'}^{p+1}P| \le \sum_{\alpha_1+\cdots+\alpha_{n+1}=p+1} \frac{(p+1)!}{(\alpha_1!)\cdots (\alpha_{n+1}!)}
	\frac{|D_{x'}^{\alpha_1}u_n|}{x_n}\cdots \frac{|D_{x'}^{\alpha_{n+1}}u_n|}{x_n} \\
	&\le \tilde{C_0} C_1^{p-1} (p-1)! (r-|x'|)^{-(p-1)}.
\end{aligned}\end{equation}
Hence, similarly as Lemma \ref{lemma3.3}, we can prove
\begin{equation}\label{4.35}
	\|N_{p+1}\|_{L^{\infty}(B_{\rho}(x_0))} \le C_2 C_1^{p-2} (p-1)! (r-|x'_0|)^{p-2}
	+ C_2 (p+1) \|D^2D_{x'}^{p}w\|_{L^{\infty}(B_{\rho}(x_0))}.
\end{equation}
Therefore, it follows from \eqref{4.31}, \eqref{4.32} and \eqref{4.35} that
\begin{equation}
	|DD^{p+1}_{x'}w(x_0)|\le c_7 C_0 C_1^{p-2}(p-1)!(r-|x'_0|)^{-(p-1)} + c_7(p+1)\rho\|D^2D_{x'}^{p}w\|_{L^{\infty}(B_{\rho}(x_0))},
\end{equation}
Which, together with  (4.27)  for $l=p$,  yields
\begin{equation}\label{iteration}
	|D^2D^p_{x'}w(x_0)|\le c_7 C_0 C_1^{p-2}(p-1)!(r-|x'_0|)^{-(p-1)} + c_7(p+1)\rho\|D^2D_{x'}^{p}w\|_{L^{\infty}(B_{\rho}(x_0))}.
\end{equation}
	
To complete  Part 4,   we need an iteration lemma:    Lemma 2 in \cite{MR0118970} or Lemma 2.2 in \cite{2018The}.
\begin{Lemma}\label{lemma4.1}
	Given an integer $p\ge 1$ and constants $0\le \varepsilon \le 1, M > 0$. Assume $h(t)$ is a positive monotone increasing function on $[0,r]$ and satisfies
	\begin{equation}
		h(t) \le \varepsilon h(t + p^{-1}(r-t)) + M(r-t)^{-p},~\forall~t\in [0,r).
	\end{equation}
	Then there exists a positive constant $C$ depending only on $\varepsilon$ such that
	\begin{equation}
		h(t)\le CM(r-t)^{-p},~\forall~t\in [0,r).
	\end{equation}
\end{Lemma}
Back to our previous proof, we note that $c_7(p+1)\rho \le c_7 r$ and denote $\varepsilon = c_7 r$. Moreover, we define the function 
\begin{equation}
	h(t) = \sup\{|D^2D_{x'}^p u|: |x'|\le t, x_n \ge \frac{1}{p} (r-|x'|)\}, \ \  \forall t\in [0,r].
\end{equation}
Combining the results in Part 3 with \eqref{iteration}, we obtain
\begin{equation}\begin{aligned}
	h(t) &\le \varepsilon h(t+\rho) + c_7 C_0 C_1^{p-2}(p-1)!(r-t)^{-(p-1)}\\
	&\le \varepsilon h(t+(p-1)^{-1}(r-t)) + c_7 C_0 C_1^{p-2}(p-1)!(r-t)^{-(p-1)}.
\end{aligned}\end{equation}
Then it follows from Lemma \ref{lemma4.1} that
\begin{equation}
	h(t) \le C(\varepsilon) c_7 C_0 C_1^{p-2}(p-1)!(r-t)^{-(p-1)}.
\end{equation}
By definition of $h(t)$, we obtain that for any $x= (x',x_n)\in G_r$ and $x_n\ge \frac{1}{p}(r-|x'|)$,
\begin{equation}
	|D^2D^p_{x'}w(x)|\le C(\epsilon) c_7 C_0 C_1^{p-2}(p-1)!(r-|x'|)^{-(p-1)}
	\le C_0 C_1^{p-1}(p-1)!(r-|x'|)^{-(p-1)}
\end{equation}
if we take $C_1 \ge C(\epsilon) c_7$. This shows Part 4 and finishes the proof of  Theorem \ref{TAnalytic}.

	%
	%
	%
	%
	%
	%
	%
	%
	%
	%
	%
	%
%section 5	
\section {iteration of ordinary differential equations along the normal direction}\label{sec5}
In previous sections, we obtain the estimates of tangential derivatives up to arbitrary order and normal derivatives only up to order 2.  In this section, we introduce an iteration about a series of ordinary differential equations along the normal direction, which originates  from \cite{MR3024302,MR3705680}. The iteration will be used to estimate the  higher order normal derivatives.

First, we differentiate the first equation in (1.3) with respect to $x_n$ and obtain
\begin{equation}
	U^{ij}u_{nij} = (n+2) \Big( \frac{D_n u}{x_n} \Big)^{n+1} \frac{x_n u_{nn}-u_n}{x_n^2}.
\end{equation}
We rewrite it as
\begin{equation}\label{5.2}
	u_{nnn} = (n+2) (U^{nn})^{-1}\Big( \frac{D_n u}{x_n} \Big)^{n+1} \frac{x_n u_{nn}-u_n}{x_n^2}
	-(U^{nn})^{-1}\sum_{i+j<2n} U^{ij}u_{nij}.
\end{equation}
Define the functions $f_i, g_i $ and $  h_i$  $(i=0, 1)$ in $\mathbb{R}^n_+$,
\begin{equation}\label{5.3}\begin{aligned}
	&f_0 = u_n, \; f_1=D_n f_0 - \frac{f_0}{x_n},\\
	& g_0 = D_n f_0 - h_0 \frac{f_0}{x_n}, \; g_1=-(U^{nn})^{-1}\sum_{i+j<2n} U^{ij}u_{nij},\\
	&h_0 = (n+2) (U^{nn})^{-1}\Big( \frac{D_n u}{x_n} \Big)^{n+1},\; h_1 = h_0 - 1.
\end{aligned}\end{equation}
Note that
\begin{equation}\label{5.4}
	\frac{x_n u_{nn}-u_n}{x_n^2}=D_n\Big(\frac{u_n}{x_n}\Big)=D_n\Big(\frac{f_0}{x_n}\Big)=\frac{D_n f_0}{x_n}-\frac{f_0}{x_n^2}=\frac{f_1}{x_n}
\end{equation}
and
\begin{equation}\label{5.5}
	u_{nnn}=D_{n} (D_n f_0) = D_n\Big(f_1+\frac{f_0}{x_n}\Big)=D_n f_1 + \frac{f_1}{x_n}.
\end{equation}
It follows from \eqref{5.2} that
\begin{equation*}
	D_n f_1 + \frac{f_1}{x_n} = h_0  \frac{f_1}{x_n} +g_1,
\end{equation*}
which yields
\begin{equation}\label{5.6}
	D_n f_1 =  h_1 \frac{f_1}{x_n} +g_1.
\end{equation}
Similarly for $2\le k\le n+2$, we define
\begin{equation}\begin{aligned}
	&f_k=D_n f_{k-1} - \frac{f_{k-1}}{x_n},\\
	&g_k=D_n g_{k-1}+D_n h_{k-1} \frac{f_{k-1}}{x_n},\\
	&h_k = h_{k-1}-1=h_0-k.
\end{aligned}\end{equation}
As a generalization of \eqref{5.4}-\eqref{5.6}, we claim the following property.
\begin{Lemma}\label{lemma5.1}
	 For any integer $0\le k\le n+2$ and $x=(x',x_n)\in \mathbb{R}^n_+$, we have 
	\begin{align}
		\label{5.8}&\frac{f_k}{x_n}=D_n \Big(\frac{f_{k-1}}{x_n}\Big)=D_n^k \Big(\frac{f_0}{x_n}\Big)=D_n^k \Big(\frac{u_n}{x_n}\Big);\\
		\label{5.9}&D_n^{k+2} u = D_n f_k + k \frac{f_k}{x_n};\\
		\label{5.10}&D_n f_k =h_k \frac{f_k}{x_n} +g_k.
	\end{align}
\end{Lemma}
\begin{proof}
	It is easy to see that \eqref{5.8}-\eqref{5.10} hold for $k=0$. Moreover, we obtain by computation that for any integer $0\le k\le n+1$ and $x=(x',x_n)\in \mathbb{R}^n_+$,
	\begin{align*}
		\frac{f_{k+1}}{x_n}&= \frac{D_n f_k}{x_n}-\frac{f_k}{x_n^2}=D_n \Big(\frac{f_k}{x_n}\Big),\\
		D_n f_{k+1} + (k+1) \frac{f_{k+1}}{x_n}&=D_n \Big(f_{k+1} + (k+1) \frac{f_k}{x_n}\Big) = D_n \Big(D_n f_k + k \frac{f_k}{x_n}\Big),\\
		D_n f_{k+1}-h_{k+1} \frac{f_{k+1}}{x_n} -g_{k+1}&= D_n \Big(D_n f_k - \frac{f_k}{x_n}\Big) - (h_k-1)D_n \Big(\frac{f_k}{x_n}\Big)- D_n g_k - D_n h_k \frac{f_k}{x_n} \\
		&=D_n\Big(D_n f_k -h_k \frac{f_k}{x_n} -g_k\Big).
	\end{align*}
	Then, (5.8) follows directly from the first equality, while (5.9) and (5.10)  can be proved by last two ones and  simple induction method.
\end{proof}

Recalling that $n$ is even and $u\in C^{\infty}(\overline{\mathbb{R}^n_+})$, we obtain by \eqref{5.8} that $\lim\limits_{x_n\to 0+} f_k(x,x_n)$ exists for any $x'\in \mathbb{R}^{n-1}$ and $0\le k\le n+2$. Then we define $f_k(x',0)=\lim\limits_{x_n\to 0+} f_k(x,x_n)$ and so as $g_k, h_k$. Note that  $f_0(x',0)=u_n(x',0)=0$ by Lemma \ref{lemma2.1}. It follows from the definitions of $f_k, g_k, h_k$ that all of them are global smooth in $\overline{\mathbb{R}^n_+}$ (see \cite{MR3705680}).

By \eqref{5.3} and Lemma \eqref{2.2}, we have $h_0(x',0)=n+2$, which yields $h_k(x',0)=n+2-k$. Thus, for $1\le k\le n+1$, $\frac{h_k}{x_n}$ is singular near $\{x_n=0\}$ and then \eqref{5.10} is a singular ordinary differential equation about $x_n$. Instead for $k = n+2$, the singularity disappears since $h_{n+2}(x',0)=0$. It also shows one reason why we introduce the iteration of $f_k,g_k,h_k$.

In order to estimate $D^{k+2} u$, it is sufficient to consider the estimates of derivatives for $f_k$ due to \eqref{5.9} in Lemma \ref{lemma5.1}. For this purpose, we need a different form of \eqref{5.10}, which is given in the following lemma.
\begin{Lemma}\label{lemma5.2}
	For any integer $1\le k\le n+1$, the equation \eqref{5.10} is equivalent to
	\begin{equation}\label{ode2}
		D_n f_k = (n+2-k)\frac{f_k}{x_n} + D_n^{k-1} F,
	\end{equation}
	where $F=g_1+(h_0-n-2)f_1/x_n$. In particular, we have
	\begin{equation}
		g_k = D_n^{k-1}g_1 +\sum_{m=1}^{k-1} \binom{k-1}{m-1} D_n^{k-m} h_0\cdot \frac{f_m}{x_n}.
	\end{equation}
\end{Lemma}
\begin{proof}
	We will prove Lemma  5.2  by induction on $k$. First, by \eqref{5.10}, we have
	\begin{equation}
		D_n f_1 = (h_0-1) \frac{f_1}{x_n} + g_1 = (n+1)\frac{f_1}{x_n} + F.
	\end{equation}
	So \eqref{ode2} holds for $k=1$. Assume \eqref{ode2} holds for $1\le k-1\le n+1$. Differentiating this on $x_n$, we obtain
	\begin{equation}
		D_n(D_n f_{k-1}) = (n+3-k) D_n\Big(\frac{f_{k-1}}{x_n}\Big) + D_n^{k-1}F.
	\end{equation}
	It follows from the definition of $f_k$ and \eqref{5.8} that
	\begin{equation}
		D_n f_k = D_n\Big(D_n f_{k-1}-\frac{f_{k-1}}{x_n}\Big) = (n+3-k-1) D_n\Big(\frac{f_{k-1}}{x_n}\Big) + D_n^{k-1}F=(n+2-k)\frac{f_k}{x_n} + D_n^{k-1} F.
	\end{equation}
	Thus, \eqref{ode2} is proved for any $1\le k\le n+2$. by induction. Moreover, by comparing \eqref{5.10} and \eqref{ode2}, we obtain
	\begin{equation}\begin{aligned}
		g_k &= D_n^{k-1} F + (n+2 -k - h_k)\frac{f_k}{x_n}\\
		&=D_n^{k-1}g_1 +(h_0-n-2)D_n^{k-1}\Big(\frac{f_1}{x_n}\Big) + \sum_{m=1}^{k-1} \binom{k-1}{m-1} (D_n^{k-m} h_0) D_n^{m-1} \Big(\frac{f_1}{x_n}\Big) \\
		&+ (n+2 -k - h_0+k)\frac{f_k}{x_n}\\
		&=D_n^{k-1}g_1 +\sum_{m=1}^{k-1} \binom{k-1}{m-1} (D_n^{k-m} h_0) \frac{f_m}{x_n},
	\end{aligned}\end{equation}
	where we have used $F=g_1+(h_0-n-2)f_1/x_n, h_k = h_0 - k$ and \eqref{5.8}.
\end{proof}

Now, we can use Lemma \ref{lemma5.2} to obtain a formal solution to equation (5.11) on$f_k$.
\begin{Lemma}\label{lemma5.3}
	For any integer $1\le k\le n+1$ and $x\in G_r$, we have
	\begin{equation}\label{solve}
		f_k(x',x_n) = \frac{f_k(x',r)}{r^{n+2-k}}x_n^{n+2-k}-x_n^{n+2-k} \int_{x_n}^r \frac{D_n^{k-1}F(x',s)}{s^{n+2-k}}ds.
	\end{equation}
\end{Lemma}
\begin{proof}
It follows from \eqref{ode2} that
\begin{equation}
	D_n\Big(\frac{f_k}{x_n^{n+2-k}}\Big) = \frac{1}{x_n^{n+2-k}}\Big[D_n f_k - (n+2-k)\frac{f_k}{x_n}\Big]=\frac{D_n^{k-1}F}{x_n^{n+2-k}}.
\end{equation}
Integrating it on $[x_n, r]$ yields
\begin{equation}
	\frac{f_k(x', x_n)}{x_n^{n+2-k}} = \frac{f_k(x',r)}{r^{n+2-k}}- \int_{x_n}^r \frac{D_n^{k-1}F(x',s)}{s^{n+2-k}}ds.
\end{equation}
So \eqref{solve} is proved.
\end{proof}

We will see that Lemma \ref{lemma5.3} is important when we estimate the tangential derivatives of $f_k$ and $D^{k+2}u$, although it only gives a formal expression of $f_k$. We will finish the work in the next section. As preparation, we need more accurate estimates than those in the previous sections. Now by Theorem \ref{TAnalytic} and the analyticity of $\varphi$, we may assume there exists $0<r<1$ and $C_0, C_1 >0$ depending on $u,\varphi,n$ 
and $r$ such that for any nonnegative integer $l$ and $x \in G_r$,
\begin{align}
	\label{5.20}&|D_{x'}^l D_n u| \le C_0 C_1^{(l-4)^+} (l-4)^+! x_n,\\
	\label{5.21}&|D^2 D_{x'}^l u| \le C_0 C_1^{(l-3)^+} (l-3)^+!.
\end{align}

 Using   \eqref{5.20}-\eqref{5.21} and Lemma \ref{lemma2.4}, similarly as Lemma \ref{lemma3.2},  one can  prove the following estimates.
\begin{Lemma}\label{lemma5.4}
	There exists a constant $\tilde{C_0}$ depending only on $u, \varphi, r, n$ such that for any indices $1\le i,j \le n$, nonnegative integer $l$ and $x\in G_r$,
	\begin{align}
		\label{P1}&\Big|D^l_{x'} \Big(\frac{u_n}{x_n}\Big)^{n+1}\Big|\le \tilde{C_0} C_1^{(l-4)^+} (l-4)^+!,\\
		\label{U1}&|D^l_{x'} U^{ij}| \le \tilde{C_0} C_1^{(l-3)^+} (l-3)^+!.
	\end{align}
	Moreover, if $1\le i \le n-1$, we have in $G_r$,
	\begin{equation}\label{U2}
		|D^l_{x'} U^{ni}| \le \tilde{C_0} C_1^{(l-3)^+} (l-3)^+! x_n.
	\end{equation}
\end{Lemma}
\begin{proof}
	\eqref{P1} and \eqref{U1} are directly from the proof in Lemma \ref{lemma3.2}. So we only need to prove \eqref{U2}. By the definition of $U^{ni}$, we have
	\begin{equation}
		U^{ni} = \sum_{k=1}^{n-1} u_{kn} \Big(\sum u_{i_1 j_1} \cdots u_{i_{n-2}j_{n-2}}\Big),
	\end{equation}
	where $(i_1,\cdots,i_{n-2})=(1,\cdots,i-1,\widehat{i},i+1,\cdots,n-1)$ and $(j_1,\cdots,j_{n-2})$ is a permutation of the set $\{1,\cdots,k-1,\widehat{k},k+1,\cdots,n-1\}$.
	Similarly as Lemma \ref{lemma3.2}, by \eqref{5.20}, \eqref{5.21} and Lemma \ref{lemma2.4}, we obtain
	\begin{equation}\begin{aligned}
		&|D_{x'}^l U^{ni}| = \Big|D_{x'}^l \sum_{k=1}^{n-1} u_{kn} \Big(\sum u_{i_1 j_1} \cdots u_{i_{n-2}j_{n-2}}\Big)\Big|\\
		&\le (n-1)!\sum_{\alpha_1 +\cdots + \alpha_{n-1}=l} \frac{l!}{\alpha_1!\cdots \alpha_{n-1}!}
		|D_{x'}^{\alpha_1+1}D_n u||D^2 D_{x'}^{\alpha_2} u| \cdots |D^2 D_{x'}^{\alpha_{n-1}} u|\\
		&\le (n-1)! \sum_{\alpha_1 +\cdots + \alpha_{n-1}=l} \frac{l!(\alpha_1-3)^+! \cdots (\alpha_{n-1}-3)^+!}{\alpha_1!\cdots \alpha_{n-1}!}
		C_0^{n-1} C_1^{(\alpha_1-3)^+ + \cdots (\alpha_{n-1}-3)^+}x_n\\
		&\le (n-1)! S[n-1,l;3]
		C_0^{n-1} C_1^{(\alpha_1 + \alpha_2 + \cdots \alpha_{n-1} - 3)^+}l!x_n\\
		&\le (n-1)![(n-1)S_3]^{n-1} C_0^{n-1} C_1^{(l-3)^+}(l-3)^+!x_n.
	\end{aligned}\end{equation}
	Choosing $\tilde{C_0} \ge (n-1)![(n-1)S_3]^{n-1} C_0^{n-1}$, we have
	\begin{equation}
		|D_{x'}^l U^{ni}| \le \tilde{C_0} C_1^{(l-3)^+}(l-3)^+!x_n.
	\end{equation}
\end{proof}
Note that $U^{nn}(0)=\det D^2_{x'} \varphi(0)=1$. We assume $0<r<1$ is small enough such that $\frac 23 < U^{nn} <\frac 43$ in $G_r$. Then we have the following estimates.
\begin{Lemma}\label{lemma5.5}
	There exist constants $C_1, C_2$ depending only on $u, \varphi, r, n$ such that for any nonnegative integer $l$ and $x\in G_r$,
	\begin{equation}\label{inverse}
		|D_{x'}^l (U^{nn})^{-1}| \le C_2 C_1^{(l-3)^+} (l-3)^+!.
	\end{equation}
\end{Lemma}
\begin{proof}
	We will prove \eqref{inverse} by induction on $l$. Fix any $x\in G_r$.
	First, by the fact $\frac 23 < U^{nn} <\frac 43$ and \eqref{U1}, it is easy to obtain that
	\begin{align*}
		&|(U^{nn})^{-1}| \le \frac 32;\\
		&|D_{x'}(U^{nn})^{-1}|=\Big|\frac{D_{x'}U^{nn}}{(U^{nn})^2}\Big|\le \frac 94 \tilde{C_0};\\
		&|D_{x'}^2 (U^{nn})^{-1}|\le\Big|\frac{D_{x'}^2 U^{nn}}{(U^{nn})^2}\Big| + 2\Big|\frac{(D_{x'} U^{nn})^2}{(U^{nn})^3}\Big|\le \frac{9}{4} \tilde{C_0} + \frac{27}{4}\tilde{C_0}^2\\
		&|D_{x'}^3 (U^{nn})^{-1}|\le\Big|\frac{D_{x'}^3 U^{nn}}{(U^{nn})^2}\Big| + 6\Big|\frac{D_{x'}^2 U^{nn} \cdot D_{x'} U^{nn}}{(U^{nn})^3}\Big|+6\Big|\frac{(D_{x'} U^{nn})^3}{(U^{nn})^4}\Big|\\
		&\le \frac{9}{4} \tilde{C_0} + \frac{81}{4}\tilde{C_0}^2 + \frac{243}{8}\tilde{C_0}^3.
	\end{align*}
	Taking $C_2\ge \max\{\frac 32, \frac{9}{4} \tilde{C_0} + \frac{81}{4}\tilde{C_0}^2 + \frac{243}{8}\tilde{C_0}^3\}$, we se that \eqref{inverse} holds for all $0\le l\le 3$. Now  for   arbitrary integer $p\ge 4$ we assume that \eqref{inverse} holds for all $0\le l \le p-1$. Thus, it is sufficient to prove the case for $l=p$. Observing that
	\begin{equation}
		0=D_{x'}^p [U^{nn}(U^{nn})^{-1}]= \sum_{m=0}^p \binom{p}{m} (D_{x'}^m U^{nn}) [D_{x'}^{p-m} (U^{nn})^{-1}],
	\end{equation}
	we obtain by the induction hypotheses, Lemma \ref{lemma2.4} and \eqref{U1} that
	\begin{equation}\begin{aligned}
		&|D_{x'}^p (U^{nn})^{-1}| \le (U^{nn})^{-1} \sum_{m=1}^{p-1} \binom{p}{m} |D_{x'}^m U^{nn}| |D_{x'}^{p-m} (U^{nn})^{-1}|+(U^{nn})^{-2}|D_{x'}^p U^{nn}|\\
		&\le \frac 32 \sum_{m=1}^{p-1} \frac{p!}{m!(p-m)!} \tilde{C_0} C_1^{(m-3)^+} (m-3)^+! C_2 C_1^{(p-m-3)^+} (p-m-3)^+! +\frac{9}{4}\tilde{C_0} C_1^{p-3} (p-3)!\\
		&  \le \frac 32 \tilde{C_0} C_2 C_1^{p-4} p! S[2, p;3] + \frac{9}{4}\tilde{C_0} C_1^{p-2} (p-2)!\\
		&\le 6 S_3^2 \tilde{C_0} C_2 C_1^{p-4} (p-3)! + \frac{9}{4}\tilde{C_0} C_1^{p-3} (p-3)!.
	\end{aligned}\end{equation}
	Choosing $C_1 \ge 12S_3^2 \tilde{C_0}$, we have
	\begin{equation}
		|D_{x'}^p (U^{nn})^{-1}| \le \frac 12 C_2 C_1^{p-3} (p-3)!+ \frac 14 C_2 C_1^{p-3} (p-3)!\le C_2 C_1^{p-3} (p-3)!.
	\end{equation}
	 This  proves  \eqref{inverse}  for all $l \ge 0$ by induction.
\end{proof}

At the end of this section, we use the above lemmas to give the estimates about $h_0$ and $g_1$.
\begin{Lemma}\label{lemma5.6}
	There exist constants $C_1, C_3$ depending only on $u, \varphi, r, n$ such that for any nonnegative integer $l$ and $x\in G_r$,
	\begin{align}
		\label{5.32}&|D_{x'}^l h_0| \le C_3 C_1^{(l-3)^+} (l-3)^+!,\\
		\label{5.33}&|D_{x'}^l g_1| \le C_3 C_1^{(l-2)^+} (l-2)^+! x_n.
	\end{align}
\end{Lemma}
\begin{proof}
	For any given $l\ge 0$ and $x\in G_r$, we prove \eqref{5.32} and \eqref{5.33} respectively.
	
	(1) By the definition of $h_0$, \eqref{P1}, \eqref{inverse} and Lemma \ref{lemma2.4}, we have
	\begin{equation}\begin{aligned}
		|D_{x'}^l h_0| &= (n+2)\Big|\sum_{\alpha=0}^l D_{x'}^{\alpha} (U^{nn})^{-1} D_{x'}^{l-\alpha} \Big(\frac{u_n}{x_n}\Big)^{n+1} \Big|\\
		&\le (n+2)\sum_{\alpha=0}^l |D_{x'}^{\alpha} (U^{nn})^{-1}| \Big|D_{x'}^{l-\alpha}  \Big(\frac{u_n}{x_n}\Big)^{n+1}\Big|\\
		&\le (n+2)\sum_{\alpha=0}^l C_2 C_1^{(\alpha-3)^+}(\alpha-3)^+! \tilde{C_0} C_1^{(l- \alpha-3)^+}(l- \alpha-3)^+!\\
		&\le (n+2)S[2,l;3]\tilde{C_0}C_2 C_1^{(l-3)^+}l!\\
		&\le 4(n+2)S_3^2 \tilde{C_0} C_2 C_1^{(l-3)^+}(l-3)^+!.
	\end{aligned}\end{equation}

	(2) By definition of $g_1$, we have
	\begin{equation}
		g_1 =  -(U^{nn})^{-1} \Big(\sum_{i,j<n} U^{ij} u_{nij} + 2\sum_{i=1}^{n-1} U^{ni}u_{nni} \Big).
	\end{equation}
	 This, together with (5.28), (5.23), (5.20) (5.24) (5.21)   and Lemma \ref{lemma2.4} in order, yields 
	\begin{equation}\begin{aligned}
		|D^l_{x'} g_1| &\le \sum_{\alpha+\beta+\gamma=l} \frac{l!}{\alpha!\beta!\gamma!} |D_{x'}^{\alpha} (U^{nn})^{-1}|\Big(\sum_{i,j<n} |D_{x'}^{\beta} U^{ij} D_{x'}^\gamma u_{nij})| + 2\sum_{i=1}^{n-1} |D_{x'}^{\beta}U^{ni} D_{x'}^\gamma u_{nni}| \Big)\\
		&\le \sum_{\alpha+\beta+\gamma=l} \frac{l!}{\alpha!\beta!\gamma!} C_2 C_1^{(\alpha-2)^+}(\alpha-2)^+! \Big(\sum_{i,j<n} \tilde{C_0} C_1^{(\beta-2)^+}(\beta-2)^+! \\
		&\cdot C_0 C_1^{(\gamma-2)^+}(\gamma-2)^+!x_n + 2\sum_{i=1}^{n-1} \tilde{C_0} C_1^{(\beta-2)^+}(\beta-2)^+!x_n C_0 C_1^{(\gamma-2)^+}(\gamma-2)^+! \Big)\\
		&\le \tilde{C_0} C_0 C_2 S[3,l;2] C_1^{(\alpha+\beta+\gamma-2)^+} l! x_n [(n-1)^2 + 2(n-1)]\\
		&\le 27(n^2-1) S_2^3 \tilde{C_0} C_0 C_2 C_1^{(l-2)^+} (l-2)^+ x_n.
	\end{aligned}\end{equation}
\end{proof}

%section 6(final)	
\section {high-order derivative estimates and analyticity of solutions}\label{sec6}
In this section, we will use the  results in Section \ref{sec5} to give the estimates of tangential derivatives of all $f_k$ for $1\le k \le n+1$. Furthermore, we will prove the analyticity of $u$ by the estimates. Note that for given $r > 0$, \eqref{Eq} is uniformly elliptic in $\{x_n\ge r\}$. From the standard theory \cite{MR1431263,MR2284971, M1, M2, M3}, $u$ is analytic on $\partial G_r \cap \{x_n = r\}$, i.e. we can choose proper constants $C_0, C_1>1$ such that for any $|x'|\le r$ and $l\ge 0, 1 \le k\le n+1$,
\begin{equation}\label{fr}
	|D_{x'}^l f_k(x',r)| \le C_0 r^2 C_1^{(l+k-3)^+}(l+k-3)^+!.
\end{equation}
Now we claim the following estimates.
\begin{Theorem}\label{NAnalytic}
	There exist constants $B_1, C_0, C_1>1$ depending only on $u,\varphi,r,n$ such that for any integers $1\le k \le n+1, l\ge k$ and $x\in G_r$,
	
	(1) if $k$ is odd,
	\begin{align}
		&|D_{x'}^{l-k} f_k| \le C_0 B_1^k C_1^{(l-3)^+}(l-3)^+! x_n^2,\\
		&|D_{x'}^{l-k} D_n^{k+2} u| \le C_0 B_1^k C_1^{(l-3)^+}(l-3)^+! x_n;
	\end{align}
	
	(2) if $k$ is even,
	\begin{align}
		&|D_{x'}^{l-k} f_k| \le C_0 B_1^k C_1^{(l-k-3)^+}(l-3)^+! x_n,\\
		&|D_{x'}^{l-k} D_n^{k+2} u| \le C_0 B_1^k C_1^{(l-k-3)^+}(l-3)^+!.
	\end{align}
\end{Theorem}
We will prove Theorem \ref{NAnalytic} by induction on $1\le k\le n + 1$. First, we prove the case for $k=1$.
\begin{Proposition}\label{f_1}
	There exist constants $B_1, C_0, C_1>1$ depending only on $u,\varphi,r,n$ such that for any integer $l\ge 0$ and $x\in G_r$,
	\begin{align}
		\label{f1}&|D_{x'}^l f_1| \le C_0 B_1 C_1^{(l-2)^+} (l-2)^+! x_n^2,\\
		\label{u3}&|D_{x'}^l D_n^3 u| \le C_0 B_1 C_1^{(l-2)^+} (l-2)^+! x_n.
	\end{align}
\end{Proposition}
\begin{proof}
	By Lemma \ref{lemma5.3} and the definition of $F$ in Lemma 5.2, we have
	\begin{equation}\begin{aligned}
		\frac{f_1(x',x_n)}{x_n^2} &= \frac{f_1(x',r)}{r^{n+1}}x_n^{n-1}-x_n^{n-1} \int_{x_n}^r \frac{F(x',s)}{s^{n+1}}ds\\
		&=\frac{f_1(x',r)}{r^{n+1}}x_n^{n-1}-x_n^{n-1} \int_{x_n}^r \frac{1}{s^{n}}\Big[\frac{g_1}{s}+(h_0-n-2)\frac{f_1}{s^2}\Big]ds.
	\end{aligned}\end{equation}
	Differentiating this, we obtain
	\begin{equation}\begin{aligned}
		&\frac{D_{x'}^l f_1(x',x_n)}{x_n^2}=\frac{D_{x'}^l f_1(x',r)}{r^{n+1}}x_n^{n-1} \\
		&- x_n^{n-1} \int_{x_n}^r \frac{1}{s^{n}}\Big[\frac{D_{x'}^l g_1}{s}+(h_0-n-2)\frac{D_{x'}^l f_1}{s^2}+\sum_{m=1}^l \binom{l}{m}(D_{x'}^m h_0)\frac{D_{x'}^{l-m}f_1}{s^2} \Big] ds.
	\end{aligned}\end{equation}
	Considering the $L^\infty$-norms and using the fact $\tilde{x_n} < r < 1$, we have
	\begin{equation}\begin{aligned}
		&\frac{|D_{x'}^l f_1(x',\tilde{x_n})|}{\tilde{x_n}^2} \le \frac{|D_{x'}^lf_1(x',r)|}{r^{n+1}}\tilde{x_n}^{n-1}+\tilde{x_n}^{n-1} \int_{\tilde{x_n}}^r \frac{1}{s^{n}}ds\Big[\Big\|\frac{D_{x'}^l g_1}{x_n}\Big\|_{L^\infty(G_r)}\\
		&+\|h_0-n-2\|_{L^\infty(G_r)}\Big\|\frac{D_{x'}^l f_1}{x_n^2}\Big\|_{L^\infty(G_r)}+\sum_{m=1}^l \binom{l}{m}\|D_{x'}^m h_0\|_{L^\infty(G_r)} \Big\|\frac{D_{x'}^{l-m}f_1}{x_n^2}\Big\|_{L^\infty(G_r)} \Big]\\
		&\le \frac{|D_{x'}^lf_1(x',r)|}{r^2}+ \frac{1}{n-1} \Big[\Big\|\frac{D_{x'}^l g_1}{x_n}\Big\|_{L^\infty(G_r)}+\|h_0-n-2\|_{L^\infty(G_r)}\Big\|\frac{D_{x'}^l f_1}{x_n^2}\Big\|_{L^\infty(G_r)}\\
		&+\sum_{m=1}^l \binom{l}{m} \|D_{x'}^m h_0\|_{L^\infty(G_r)} \Big\|\frac{D_{x'}^{l-m}f_1}{x_n^2}\Big\|_{L^\infty(G_r)} \Big].\\
	\end{aligned}\end{equation}
	Note that $h_0(x',0)=n+2$, so we choose $0<r<1$ small enough such that $\|h_0-n-2\|_{L^\infty(G_r)}<\frac 12$. Then we obtain
	\begin{equation}\label{f1e}\begin{aligned}
		\Big\|\frac{D_{x'}^l f_1}{x_n^2}\Big\|_{L^\infty(G_r)} &\le 2 \frac{\|D_{x'}^l f_1(\cdot,r)\|_{L^\infty(B'_r)}}{r^2}
		+2\Big\|\frac{D_{x'}^l g_1}{x_n}\Big\|_{L^\infty(G_r)}\\
		&+2\sum_{m=1}^l \binom{l}{m} \|D_{x'}^m h_0\|_{L^\infty(G_r)} \Big\|\frac{D_{x'}^{l-m}f_1}{x_n^2}\Big\|_{L^\infty(G_r)}.
	\end{aligned}\end{equation}
	Now we prove \eqref{f1} by induction on $l$. First, if $l=0$, it follows from (6.11), (6.1) and  \eqref{5.33} in order that
	\begin{equation}\begin{aligned}
		\Big\|\frac{f_1}{x_n^2}\Big\|_{L^\infty(G_r)} &\le 2 \frac{\|f_1(\cdot,r)\|_{L^\infty(B'_r)}}{r^2}
		+2\Big\|\frac{g_1}{x_n}\Big\|_{L^\infty(G_r)}\\
		&\le 2C_0 + 2 C_3= 2C_0 (1+C_3/C_0).
	\end{aligned}\end{equation}
	If $l=1$, similarly, we use \eqref{5.32}, \eqref{5.33}, \eqref{fr} and \eqref{f1e} to obtain
	\begin{equation}\begin{aligned}
		\Big\|\frac{D_{x'} f_1}{x_n^2}\Big\|_{L^\infty(G_r)} &\le 2 \frac{\|D_{x'} f_1(\cdot,r)\|_{L^\infty(B'_r)}}{r^2}
		+2\Big\|\frac{D_{x'} g_1}{x_n}\Big\|_{L^\infty(G_r)}\\
		&+2\|D_{x'} h_0\|_{L^\infty(G_r)} \Big\|\frac{f_1}{x_n^2}\Big\|_{L^\infty(G_r)}\\
		&\le 2C_0 + 2C_3 + 2 C_3(2C_0 + 2C_3)\\
		&= 2C_0(1 + C_3/C_0)(1 + 2 C_3).
	\end{aligned}\end{equation}
	If $l=2$, similarly, we have
	\begin{equation}\begin{aligned}
			&\Big\|\frac{D_{x'}^2 f_1}{x_n^2}\Big\|_{L^\infty(G_r)} \le 2 \frac{\|D_{x'}^2 f_1(\cdot,r)\|_{L^\infty(B'_r)}}{r^2}
			+2\Big\|\frac{D_{x'} g_1}{x_n}\Big\|_{L^\infty(G_r)}\\
			&+4\|D_{x'} h_0\|_{L^\infty(G_r)} \Big\|\frac{D_{x'} f_1}{x_n^2}\Big\|_{L^\infty(G_r)}
			+2\|D_{x'}^2 h_0\|_{L^\infty(G_r)} \Big\|\frac{f_1}{x_n^2}\Big\|_{L^\infty(G_r)}\\
			&\le 2C_0 + 2C_3 + 4 C_3(2C_0 + 2C_3)(1+2C_3) + 2C_3 (2C_0+2C_3)\\
			&= 2C_0(1 + C_3/C_0)(1 + 2 C_3)(1 + 4 C_3).
	\end{aligned}\end{equation}
	Therefore by choosing $B_1 \ge 2(1 + C_3/C_0)(1 + 2 C_3)(1 + 4 C_3)$, we see that \eqref{f1} holds for all $0\le l\le 2$. Now we take any $p\ge 3$ and then assume \eqref{f1} holds for all $0\le l \le p-1$. Thus by induction, we only need to prove \eqref{f1} for the case $l=p$. By (6.11),   \eqref{5.33},  \eqref{5.32},  (6.1) and Lemma \ref{lemma2.4} in order, we compute
	\begin{equation*}\begin{aligned}
		\Big\|\frac{D_{x'}^p f_1}{x_n^2}\Big\|_{L^\infty(G_r)} &\le 2 \frac{\|D_{x'}^p f_1(\cdot,r)\|_{L^\infty(B'_r)}}{r^2}+2\Big\|\frac{D_{x'}^p g_1}{x_n}\Big\|_{L^\infty(G_r)}\\
		+2&\sum_{m=1}^{p-1}\binom{p}{m} \|D_{x'}^m h_0\|_{L^\infty(G_r)} \Big\|\frac{D_{x'}^{p-m}f_1}{x_n^2}\Big\|_{L^\infty(G_r)}
		+\|D_{x'}^p h_0\|_{L^\infty(G_r)} \Big\|\frac{f_1}{x_n^2}\Big\|_{L^\infty(G_r)}\\
		&\le 2C_0 C_1^{p-2}(p-2)! + 2C_3 C_1^{p-2}(p-2)! \\
		&+ 2\sum_{m=1}^{p-1}\binom{p}{m}C_3 C_1^{(m-2)^+} (m-2)^+! C_0 B_1 C_1^{(p-m-2)^+}(p-m-2)^+!\\
		&+ 2 C_3 C_1^{p-2}(p-2)! (2C_0 +2 C_3)\\
		&\le 2(C_0 +C_3)(1+2 C_3)C_1^{p-2}(p-2)!+2 C_3 C_0 B_1 C_1^{p-3} p!S[2,p;2]\\
		&\le 2(1 +C_3/C_0)(1+2 C_3)C_0 C_1^{p-2}(p-2)!+ 8 S_2^2 C_3 C_0 B_1 C_1^{p-3} (p-2)!.
	\end{aligned}\end{equation*}
	Note that $B_1 \ge 2(1 + C_3/C_0)(1+2C_3)(1+4C_3)\ge 4(1+C_3/C_0)(1+2C_3)$. We choose $C_1 \ge 16S_2^2 C_3$ and then obtain
	\begin{equation}
		\Big\|\frac{D_{x'}^p f_1}{x_n^2}\Big\|_{L^\infty(G_r)}
		\le \frac12 C_0 B_1 C_1^{p-2}(p-2)! + \frac12 C_0 B_1 C_1^{p-2}(p-2)!=C_0 B_1 C_1^{p-2}(p-2)!.
	\end{equation}
	This proves \eqref{f1}.

	Finally, we prove \eqref{u3}. By \eqref{5.9} and \eqref{5.10}, we obtain
	\begin{equation}
		D^3_n u = (h_1 \frac{f_1}{x_1} + g_1) + \frac{f_1}{x_1} = h_0 \frac{f_1}{x_n}+g_1,
	\end{equation}
	which, together with \eqref{f1} and Lemma \ref{lemma5.6}, yields
	\begin{equation}\begin{aligned}
		&\Big|\frac{D_{x'}^l D^3_n u}{x_n}\Big| \le \sum_{m=0}^l \binom{l}{m} |D_{x'}^m h_0|\Big|\frac{D_{x'}^{l-m}f_1}{x_n^2}\Big| + \Big|\frac{D_{x'}^l g_1}{x_n}\Big|\\
		&\le \sum_{m=0}^l \binom{l}{m} C_3 C_1^{(m-2)^+} (m-2)^+! C_0 B_1 C_1^{(l-m-2)^+}(l-m-2)^+! + C_3 C_1^{(l-2)^+}(l-2)^+!\\
		&\le 4 S_2^2 C_3 C_0 B_1 C_1^{(l-2)^+}(l-2)^+ + C_3 C_1^{(l-2)^+}(l-2)^+\\
		&\le (4 S_2^2 C_3+1) C_0 B_1 C_1^{(l-2)^+}(l-2)^+.
	\end{aligned}\end{equation}
	Thus,  replacing $B_1$ by $(4 S_2^2 C_3+1)B_1$, we obtain (6.7).
\end{proof}
	
Now we have found proper constants such that Theorem \ref{NAnalytic} holds for $k=1$. Then we assume $B_1,C_0,C_1$ are choosen such that Theorem \ref{NAnalytic} holds for $k = 1,2,\cdots,q-1$, where $2\le q\le n+1$. Thus, it is sufficient to prove Theorem \ref{NAnalytic} for $k=q$. For this purpose, we need to consider two cases: $q$ is even or odd.

\begin{Proposition}\label{evenq}
	Under the induction hypotheses that Theorem 6.1 holds for $k=1, 2, \cdots , q-1$, if $2 \le q \le n+1$ and $q$ is even, there exists constants $B_1,C_0,C_1>1$ 
depending only on $u,\varphi,r,n$ such that for any integer $l\ge q$ and $x\in G_r$,
	\begin{align}
		\label{6.18}&|D_{x'}^{l-q} f_q| \le C_0 B_1^q C_1^{(l-3)^+}(l-3)^+! x_n,\\
		\label{6.19}&|D_{x'}^{l-q} D_n^{q+2} u| \le C_0 B_1^q C_1^{(l-3)^+}(l-3)^+!.
	\end{align}
\end{Proposition}
\begin{proof}
	By Lemma \ref{lemma5.3}, we have
	\begin{equation}
		\frac{D_{x'}^{l-q}f_q(x',x_n)}{x_n} = \frac{D_{x'}^{l-q}f_q(x',r)}{r^{n+2-q}}x_n^{n+1-q}-x_n^{n+1-q} \int_{x_n}^r \frac{D_{x'}^{l-q}D_n^{q-1} F(x',s)}{s^{n+2-q}}ds.
	\end{equation}
	 Using the iteration in (5.8) we see that
	\begin{equation}\begin{aligned}
		&D_{x'}^{l-q} D_n^{q-1} F = D_{x'}^{l-q} D_n^{q-1}[g_1+(h_0-n-2)\frac{f_1}{x_n}]\\
		&=D_{x'}^{l-q} D_n^{q-1} g_1 + \sum_{\alpha=0}^{l-q} \sum_{\beta=0}^{q-1} \binom{l-q}{\alpha} \binom{q-1}{\beta} D_{x'}^{\alpha} D_n^{\beta} (h_0-n-2)\cdot D_{x'}^{l-q-\alpha} D_n^{q-1-\beta} \Big(\frac{f_1}{x_n}\Big)\\
		&=D_{x'}^{l-q} D_n^{q-1} g_1 + (h_0-n-2)\frac{D_{x'}^{l-q}f_q}{x_n}
		+ \sum_{\substack{0\le \alpha \le l-q, \\0 \le \beta \le q-1,\\ \alpha + \beta \ge 1}} \binom{l-q}{\alpha} \binom{q-1}{\beta}D_{x'}^{\alpha} D_n^{\beta} h_0 \cdot  \frac{D_{x'}^{l-q-\alpha}f_{q-\beta}}{x_n}.
	\end{aligned}\end{equation}
	Observing the fact $n+2-q \ge 2$, which follows from the assumption that   $n$ and $q$ are even, 
	By (6.20), (6.21) and the estimate $\|h_0-n-2\|_{L^\infty(G_r)}<\frac 12$ before \eqref{f1e}, we obtain
	\begin{equation}\begin{aligned}
		\frac{|D_{x'}^{l-q}f_q(x',\tilde{x_n})|}{\tilde{x_n}} &\le \frac{|D_{x'}^{l-q}f_q(x',r)|}{r^{n+2-q}}\tilde{x_n}^{n+1-q}+\tilde{x_n}^{n+1-q} \int_{\tilde{x_n}}^r \frac{1}{s^{n+2-q}}\Big\|D_{x'}^{l-q}D_n^{q-1} F\Big\|_{L^\infty(G_r)}ds\\
		&\le \frac{|D_{x'}^{l-q}f_q(x',r)|}{r}+\frac{1}{n+1-q} \Big\|D_{x'}^{l-q}D_n^{q-1} F\Big\|_{L^\infty(G_r)}\\
		&\le \frac{|D_{x'}^{l-q}f_q(x',r)|}{r}+ \Big\|D_{x'}^{l-q}D_n^{q-1} g_1\Big\|_{L^\infty(G_r)}
		+\frac 12 \Big\|\frac{D_{x'}^{l-q}f_q}{x_n}\Big\|_{L^\infty(G_r)}\\
		&+ \sum_{\substack{0\le \alpha \le l-q, 0 \le \beta \le q-1\\ \alpha + \beta \ge 1}} \binom{l-q}{\alpha} \binom{q-1}{\beta} \Big\|D_{x'}^{\alpha} D_n^{\beta} h_0 \Big\|_{L^\infty(G_r)}  \Big\|\frac{D_{x'}^{l-q-\alpha}f_{q-\beta}}{x_n}\Big\|_{L^\infty(G_r)}.
	\end{aligned}\end{equation}
	Thus,
	\begin{equation}\begin{aligned}\label{fqeven}
		&\frac 12 \Big\|\frac{D_{x'}^{l-q}f_q}{x_n}\Big\|_{L^\infty(G_r)} \le \frac{\|D_{x'}^{l-q}f_q(\cdot,r)\|_{L^\infty(B'_r)}}{r}+ \Big\|D_{x'}^{l-q}D_n^{q-1} g_1\Big\|_{L^\infty(G_r)}\\
		&+\sum_{\substack{0\le \alpha \le l-q, 0 \le \beta \le q-1,\\ \alpha + \beta \ge 1}} \binom{l-q}{\alpha} \binom{q-1}{\beta} \Big\|D_{x'}^{\alpha} D_n^{\beta} h_0 \Big\|_{L^\infty(G_r)}  \Big\|\frac{D_{x'}^{l-q-\alpha}f_{q-\beta}}{x_n}\Big\|_{L^\infty(G_r)}.
	\end{aligned}\end{equation}
	To prove \eqref{6.18}, it is sufficient to estimate the terms on the right side of \eqref{fqeven}. 

Summing up 
	\eqref{5.20}, \eqref{5.21} and the induction hypotheses for Theorem 6.1, we see that for any $0\le k \le q-1, l\ge k$ and $x\in G_r$,
	\begin{equation}\label{6.24}
		|D_{x'}^{l-k} f_k| \le C_0 B_1^k C_1^{(l-3)^+}(l-3)^+! x_n
	\end{equation}
	and for any $0\le k \le q+1, l\ge k$ and $x\in G_r$,
	\begin{equation}\label{6.25}
		|D_{x'}^{l-k} D_n^k u| \le C_0 B_1^{(k-2)^+} C_1^{(l-5)^+}(l-5)^+!.
	\end{equation}
	Now we fix $1\le i,j\le n-1$, $0\le k \le q-1$ and then consider the derivatives of $U^{ij},U^{ni}$ and $U^{nn}$. By the definition of determinants, we need to compute
	\begin{equation}\begin{aligned}
		|D_{x'}^{l}D_n^k (u_{i_1 j_1}\cdots u_{i_{n-1} j_{n-1}})|&\le \sum_{\substack{\alpha_1+\cdots+\alpha_{n-1}=l,\\\beta_1+\cdots+\beta_{n-1}=k}}
		\frac{l!k!}{\alpha_1!\cdots \alpha_{n-1}!\beta_1!\cdots \beta_{n-1}!}\\
		&\cdot |D_{x'}^{\alpha_1}D_n^{\beta_1}u_{i_1 j_1}|\cdots|D_{x'}^{\alpha_{n-1}}D_n^{\beta_{n-1}}u_{i_{n-1} j_{n-1}}|.
	\end{aligned}\end{equation}
	Note that in the case for $U^{ij}$, $n$ appears exactly two times among the indices $i_m, j_m$ for $1\le m\le n-1$. As for $U^{ni}$ and $U^{nn}$, $n$ appears exactly one time and zero times respectively. So by \eqref{6.25} we obtain
	\begin{equation}\begin{aligned}
		|D_{x'}^{l}D_n^k U^{ij}|&\le (n-1)! \sum_{\substack{\alpha_1+\cdots+\alpha_{n-1}=l,\\\beta_1+\cdots+\beta_{n-1}=k}}
		\frac{l!k!}{\alpha_1!\cdots \alpha_{n-1}!\beta_1!\cdots \beta_{n-1}!}\\
		&\cdot C_0^{n-1}B_1^{(k+2-2)^+} C_1^{(l+k-3)^+}(\alpha_1+\beta_1-3)^+!\cdots (\alpha_{n-1}+\beta_{n-1}-3)^+!;
	\end{aligned}\end{equation}
	\begin{equation}\begin{aligned}
		|D_{x'}^{l}D_n^k U^{ni}|&\le (n-1)! \sum_{\substack{\alpha_1+\cdots+\alpha_{n-1}=l,\\\beta_1+\cdots+\beta_{n-1}=k}}
		\frac{l!k!}{\alpha_1!\cdots \alpha_{n-1}!\beta_1!\cdots \beta_{n-1}!}\\
		&\cdot C_0^{n-1}B_1^{(k+1-2)^+} C_1^{(l+k-3)^+}(\alpha_1+\beta_1-3)^+!\cdots (\alpha_{n-1}+\beta_{n-1}-3)^+!;
	\end{aligned}\end{equation}
	\begin{equation}\begin{aligned}
		|D_{x'}^{l}D_n^k U^{nn}|&\le (n-1)! \sum_{\substack{\alpha_1+\cdots+\alpha_{n-1}=l,\\\beta_1+\cdots+\beta_{n-1}=k}}
		\frac{l!k!}{(\alpha_1)!\cdots (\alpha_{n-1})!(\beta_1)!\cdots (\beta_{n-1})!}\\
		&\cdot C_0^{n-1}B_1^{(k-2)^+} C_1^{(l+k-3)^+}(\alpha_1+\beta_1-3)^+!\cdots (\alpha_{n-1}+\beta_{n-1}-3)^+!.
	\end{aligned}\end{equation}
	By combinatorial identities, we have
	\begin{equation}
		\sum_{\substack{\alpha_1+\cdots+\alpha_{n-1}=l,\\\beta_1+\cdots+\beta_{n-1}=k}}
		\frac{(\alpha_1+\beta_{n-1})!\cdots (\alpha_{n-1}+\beta_{n-1})!}{\alpha_1!\cdots \alpha_{n-1}!\beta_1!\cdots \beta_{n-1}!}
		=\frac{(l+k)!}{l!k!}.
	\end{equation}
	This, together with Lemma \ref{lemma2.4}, gives us
	
	\begin{equation}\label{6.31}\begin{aligned}
		&\sum_{\substack{\alpha_1+\cdots+\alpha_{n-1}=l,\\\beta_1+\cdots+\beta_{n-1}=k}}
		\frac{l!k!(\alpha_1+\beta_1-3)^+!\cdots (\alpha_{n-1}+\beta_{n-1}-3)^+!}{\alpha_1!\cdots \alpha_{n-1}!\beta_1!\cdots \beta_{n-1}!}\\
		&= \sum_{\gamma_1+\cdots+\gamma_{n-1}=l+k} \sum_{\substack{\alpha_1+\cdots+\alpha_{n-1}=l,\\ \alpha_m \le \gamma_m,\forall m}}
		\frac{l!k!(\gamma_1-3)^+!\cdots (\gamma_{n-1}-3)^+!}{\alpha_1!\cdots \alpha_{n-1}!(\gamma_1-\alpha_1)!\cdots (\gamma_{n-1}-\alpha_1)!}\\
		&= \sum_{\gamma_1+\cdots+\gamma_{n-1}=l+k} \frac{(\gamma_1-3)^+!\cdots (\gamma_{n-1}-3)^+!}{\gamma_1!\cdots \gamma_{n-1}!}\\
		&\cdot \sum_{\substack{\alpha_1+\cdots+\alpha_{n-1}=l,\\ \alpha_m \le \gamma_m,\forall m}}
		\frac{l!k!{\gamma_1!\cdots \gamma_{n-1}!}}{\alpha_1!\cdots \alpha_{n-1}!(\gamma_1-\alpha_1)!\cdots (\gamma_{n-1}-\alpha_1)!}\\
		&= \sum_{\gamma_1+\cdots+\gamma_{n-1}=l+k} \frac{(\gamma_1-3)^+!\cdots (\gamma_{n-1}-3)^+!}{\gamma_1!\cdots \gamma_{n-1}!} (l+k)!\\
		&= S[n-1, l+k; 3](l+k)!\le [(n-1)S_3]^{n-1} (l+k-3)^+!.
	\end{aligned}\end{equation}
	Take $\tilde{C_0}\ge (n-1)![(n-1)S_3]^{n-1}C_0^{n-1}$. Then (6.27)-(6.29) yields    the derivative estimates about $U^{ij},U^{ni}$ and $U^{nn}$ as follows:
	\begin{align}
		&|D_{x'}^{l}D_n^k U^{ij}| \le \tilde{C_0} B_1^k C_1^{(l+k-3)^+} (l+k-3)^+!;\\
		&|D_{x'}^{l}D_n^k U^{ni}| \le \tilde{C_0} B_1^{(k-1)^+} C_1^{(l+k-3)^+} (l+k-3)^+!;\\
		&|D_{x'}^{l}D_n^k U^{nn}| \le \tilde{C_0} B_1^{(k-2)^+} C_1^{(l+k-3)^+} (l+k-3)^+!.
	\end{align}
	As a corollary, we claim that for any $l \ge 0$ and $0\le k \le q-1$,
	\begin{equation}\label{6.35}
		|D_{x'}^{l} D_n^k (U^{nn})^{-1}| \le C_2 B_1^{(k-2)^+} C_1^{(l+k-3)^+} (l+k-3)^+!.
	\end{equation}
	We prove it by induction on $l+k$. First, it is easy to choose proper $C_2$ such that it holds for $0\le l+k \le 3$. Then we denote an arbitrary integer $N\ge 4$ and assume that the claim holds for all $l+k \le N-1$ and $0\le k \le q-1$.
	Then we consider the case $l+k=N, 0\le k \le q-1$.
	From
	\begin{equation}
		D_{x'}^l D_n^k[(U^{nn})^{-1}U^{nn}]=0
	\end{equation}
	we obtain
	\begin{equation}\begin{aligned}
		U^{nn}D_{x'}^l D_n^k (U^{nn})^{-1} &= \sum_{\substack{0\le\alpha\le l,0\le \beta \le k,\\ \alpha+\beta \le l+k-1}} \frac{l!k!}{\alpha!(l-\alpha)!\beta!(k-\beta)!}  \\
		&\cdot D_{x'}^\alpha D_n^{\beta}(U^{nn})^{-1} \cdot D_{x'}^{l-\alpha}D_n^{k-\beta} U^{nn}.
	\end{aligned}\end{equation}
	By the induction hypothesis for (6.35), combinatorial identities and the assumption $\frac 23 \le U^{nn} \le \frac 43$, we have
	\begin{equation}\begin{aligned}
		&\frac 23 |D_{x'}^l D_n^k (U^{nn})^{-1}| \le \frac 32 \tilde{C_0} B_1^{(k-2)^+} C_1^{l+k-3}(l+k-3)!
		+ \sum_{\gamma=1}^{l+k-1} \sum_{\substack{0\le\alpha\le l,\\ \gamma-k \le \alpha \le \gamma}} \binom{l}{\alpha}\binom{k}{\beta} \\
		&\cdot  C_2 B_1^{(\gamma-\alpha-2)^+}C_1^{(\gamma-3)^+}(\gamma-3)^+!
		\tilde{C_0} B_1^{(k+\alpha-\gamma-2)^+} C_1^{(l+k-\gamma-3)^+}(l+k-\gamma-3)^+!\\
		&\le \frac 32 \tilde{C_0} B_1^{(k-2)^+} C_1^{l+k-3}(l+k-3)!
		+ \sum_{\gamma=1}^{l+k-1} C_2\tilde{C_0} B_1^{(k-2)^+} C_1^{l+k-4} \\
		&\cdot\frac{(\gamma-3)^+! (l+k-\gamma-3)^+!}{\gamma!(l+k-\gamma)!}\sum_{\substack{0\le\alpha\le l,\\ \gamma-k \le \alpha \le \gamma}} \frac{l!k!\gamma!(l+k-\gamma)!}{\alpha!(l-\alpha)!(\gamma-\alpha)!(k+\alpha-\gamma)!}\\
		&\le \frac 32 \tilde{C_0} B_1^{(k-2)^+} C_1^{l+k-3}(l+k-3)!+C_2\tilde{C_0} B_1^{(k-2)^+} C_1^{l+k-4}S[2,l+k;3](l+k)!\\
		&\le \frac 32 \tilde{C_0} B_1^{(k-2)^+} C_1^{l+k-3}(l+k-3)!+ 4S_3^2 C_2\tilde{C_0}B_1^{(k-2)^+}C_1^{l+k-4}(l+k-3)!.
	\end{aligned}\end{equation}
	 Taking $C_2\ge 9\tilde{C_0}$ and $C_1 \ge 12 S_3^2\tilde{C_0}$, we obtain
	\begin{equation}\begin{aligned}
		|D_{x'}^l D_n^k (U^{nn})^{-1}|&\le \frac 94 \tilde{C_0} B_1^{(k-2)^+} C_1^{l+k-3}(l+k-3)!+ 6S_3^2 C_2\tilde{C_0}B_1^{(k-2)^+} C_1^{l+k-4}(l+k-3)!\\
		&\le \frac 14 C_2 B_1^{(k-2)^+} C_1^{l+k-3}(l+k-3)! + \frac12 C_2 B_1^{(k-2)^+} C_1^{l+k-3}(l+k-3)!\\
		&\le C_2 B_1^{(k-2)^+} C_1^{l+k-3}(l+k-3)!.
	\end{aligned}\end{equation}
	By induction, we finish the proof of claim \eqref{6.35}. 

Similarly, we can prove for any $l\ge 0, 0\le k \le q-1$ and $x\in G_r$,
	\begin{align}
		&|D_{x'}^l D_n^k h_0| \le C_3 B_1^k C_1^{(l+k-3)^+}(l+k-3)^+!;\\
		&|D_{x'}^l D_n^k g_1| \le C_3 B_1^k C_1^{(l+k-2)^+}(l+k-2)^+!.
	\end{align}
	So by \eqref{fqeven}, the rest proof is almost the same as Proposition \ref{f_1}.
\end{proof}	
\begin{Proposition}\label{oddq}
	Under the induction hypotheses that Theorem 6.1 holds for $k=1, 2, \cdots, q-1$, if $2 \le q \le n+1$ and $q$ is odd, there exists constants $B_1,C_0,C_1>1$ depending only on $u,\varphi,r,n$ such that for any integer $l\ge q$ and $x\in G_r$,
	\begin{align}
		&|D_{x'}^{l-q} f_q| \le C_0 B_1^q C_1^{(l-3)^+}(l-3)^+! x_n^2,\\
		&|D_{x'}^{l-q} D_n^{q+2} u| \le C_0 B_1^q C_1^{(l-3)^+}(l-3)^+!x_n.
	\end{align}
\end{Proposition}
\begin{proof}
	The proof is almost the same as Proposition \ref{evenq}. We indicate that there are differences only in the case 
in which $k$ is  odd satisfying $1\le k\le q-2$. Fix $l\ge 0$ and $x\in G_R$. For example, we have for any odd $k\le q-2$,
	\begin{equation}
		|D_{x'}^l f_k| \le C_0 B_1^k C_1^{(l+k-3)^+}(l+k-3)^+! x_n^2,
	\end{equation}
	and for any odd $k \le q$,
	\begin{equation}
		|D_{x'}^l D_n^k u| \le C_0 B_1^{(k-2)^+} C_1^{(l+k-5)^+}(l+k-5)^+!x_n.
	\end{equation}
	Then, we prove for indices $1\le i,j \le n-1$ and odd $k \le q-2$,
	\begin{align}
		&|D_{x'}^{l}D_n^k U^{ij}| \le \tilde{C_0} B_1^k C_1^{(l+k-3)^+} (l+k-3)^+!x_n,\\
		&|D_{x'}^{l}D_n^k U^{ni}| \le \tilde{C_0} B_1^{(k-1)^+} C_1^{(l+k-3)^+} (l+k-3)^+!,\\
		&|D_{x'}^{l}D_n^k U^{nn}| \le \tilde{C_0} B_1^{(k-2)^+} C_1^{(l+k-3)^+} (l+k-3)^+!x_n.
	\end{align}
	As for even $k \le q-1$, we have
	\begin{align}
		&|D_{x'}^{l}D_n^k U^{ij}| \le \tilde{C_0} B_1^k C_1^{(l+k-3)^+} (l+k-3)^+!,\\
		&|D_{x'}^{l}D_n^k U^{ni}| \le \tilde{C_0} B_1^{(k-1)^+} C_1^{(l+k-3)^+} (l+k-3)^+!x_n,\\
		&|D_{x'}^{l}D_n^k U^{nn}| \le \tilde{C_0} B_1^{(k-2)^+} C_1^{(l+k-3)^+} (l+k-3)^+!.
	\end{align}
	Next, we obtain for odd $k$,
	\begin{align}
		&|D_{x'}^l D_n^k h_0| \le C_3 B_1^k C_1^{(l+k-3)^+}(l+k-3)^+! x_n,\\
		&|D_{x'}^l D_n^k g_1| \le C_3 B_1^k C_1^{(l+k-2)^+}(l+k-2)^+!.
	\end{align}
	For even $k$, we have
	\begin{align}
		&|D_{x'}^l D_n^k h_0| \le C_3 B_1^k C_1^{(l+k-3)^+}(l+k-3)^+!,\\
		&|D_{x'}^l D_n^k g_1| \le C_3 B_1^k C_1^{(l+k-2)^+}(l+k-2)^+!x_n.
	\end{align}
	At last, similarly as \eqref{fqeven}, we note that $n+1-q \ge 2$ since $n$ is even and $q$ is odd. Then we can use the estimate
	\begin{equation}
		\Big|\int_{\tilde{x_n}}^r \frac{D_{x'}^{l-q}D_n^{q-1} F}{s^{n+2-q}}ds\Big| \le
		\int_{\tilde{x_n}}^r \frac{1}{s^{n+1-q}}\Big\|\frac{D_{x'}^{l-q}D_n^{q-1} F}{x_n}\Big\|_{L^\infty(G_r)}ds.
	\end{equation}
	So we obtain
	\begin{equation}\begin{aligned}\label{fqodd}
		&\frac 12 \Big\|\frac{D_{x'}^{l-q}f_q}{x_n^2}\Big\|_{L^\infty(G_r)} \le \frac{\|D_{x'}^{l-q}f_q(\cdot,r)\|_{L^\infty(B'_r)}}{r^2}+ \Big\|\frac{D_{x'}^{l-q}D_n^{q-1} g_1}{x_n}\Big\|_{L^\infty(G_r)}\\
		&+\sum_{\substack{0\le \alpha \le l-q, 0 \le \beta \le q-1,\\ \alpha + \beta \ge 1, \beta~even}} \binom{l-q}{\alpha} \binom{q-1}{\beta} \Big\|D_{x'}^{\alpha} D_n^{\beta} h_0 \Big\|_{L^\infty(G_r)}  \Big\|\frac{D_{x'}^{l-q-\alpha}f_{q-\beta}}{x_n^2}\Big\|_{L^\infty(G_r)}\\
		&+\sum_{\substack{0\le \alpha \le l-q, 0 \le \beta \le q-1,\\ \alpha + \beta \ge 1,\beta~odd}} \binom{l-q}{\alpha} \binom{q-1}{\beta} \Big\|\frac{D_{x'}^{\alpha} D_n^{\beta} h_0}{x_n} \Big\|_{L^\infty(G_r)}  \Big\|\frac{D_{x'}^{l-q-\alpha}f_{q-\beta}}{x_n}\Big\|_{L^\infty(G_r)}.
	\end{aligned}\end{equation}
	By \eqref{fqodd}, we finish the proof as we do in the proof of Proposition  6.3.
\end{proof}
	
Therefore, with Proposition \ref{evenq} and Proposition \ref{oddq} in hand, Theorem \ref{NAnalytic} is proved by induction. As a direct corollary, we have
\begin{Corollary}\label{Corn+1}
	Let $C_0, C_1, B_1$ be in Theorem \ref{NAnalytic}. Then for any $x \in G_r$ and $l \ge 0$,
	\begin{align}
		\label{6.58}&|D_{x'}^l g_{n+1}| \le C_3 B_1^n C_1^{(l+n-2)^+}(l+n-2)^+ x_n;\\
		\label{6.59}&|D_{x'}^{l} f_{n+1}| \le C_0 B_1^{n+1} C_1^{(l+n-2)^+}(l+n-2)^+! x_n^2.
	\end{align}
\end{Corollary}

Now we can use Corollary \ref{Corn+1} to prove Theorem \ref{thm1.1}.
\begin{proof}[proof of Theorem \ref{thm1.1}]
	Denote $\xi = \frac{f_{n+1}}{x_n}$. By \eqref{5.10}, we obtain
	\[
	x_n D_n \xi + \xi=D_n(x_n \xi) = h_{n+1} \xi + g_{n+1},
	\]
	which yields
	\begin{equation}\label{xi}
		D_n \xi =\frac{h_{n+2}}{x_n} \xi + \frac{g_{n+1}}{x_n}.
	\end{equation}
	Recalling the discussions between Lemma \ref{lemma5.1} and \ref{lemma5.2}, we see that the functions 
 $h_{n+2}$, $g_{n+2}$ and $\xi = D_n f_{n+1}-f_{n+2}$ are all  global smooth in $\overline{G_r}$, 
 which, together with (6.58) and the fact  $h_{n+2} = h_0 - n-2 =0$ on $\partial G_r \cap \{x_n=0\}$, means that there are no singular terms in the equation \eqref{xi} about $\xi$. Therefore,  applying
  the classical theory of ordinary differential equations or the induction method in the proof of Theorem \ref{NAnalytic} to equation (6.60),  by the  derivative estimates in Theorem \ref{NAnalytic} 
  we see that $\xi$ is analytic in $\overline{G_r}$. It follows from \eqref{5.8} that $u$ is also analytic in $\overline{G_r}$.
\end{proof}

\end{document}